\documentclass[leqno,12pt]{article} 
\usepackage{color}
\usepackage{amsmath,amsfonts,amssymb,amsthm,amscd}
\usepackage{hyperref}
\usepackage{wrapfig}
\font\cmssl=cmss10 at 12 pt

\usepackage{amsmath}
\usepackage{amssymb}
\usepackage{amsthm}
\usepackage{pb-diagram}
\usepackage{ascmac}
\usepackage{ulem}

\newtheorem{theorem}{Theorem}[section]
\newtheorem{proposition}[theorem]{Proposition}
\newtheorem{lemma}[theorem]{Lemma}
\newtheorem{corollary}[theorem]{Corollary}
\newtheorem{remark}[theorem]{Remark}
\newtheorem{example}[theorem]{Example}
\newtheorem{definition}[theorem]{Definition}

\begin{document}
\title{
Quaternionic complex manifolds and \\ fixed-point sets of $S^{1}$-actions
}
\author{%
Kazuyuki Hasegawa
}

\maketitle
\begin{abstract}
In this paper, we study fixed-point sets of $S^{1}$-actions and compatible complex structures 
on quaternionic manifolds.
We obtain an equation involving the first Chern classes of the fixed-point set and of
a quaternionically flat manifold with compatible complex structure of closed type. 
In addition, if the first Chern class of the fixed-point set is not trivial, then 
the quaternionic manifold 
does not admit hypercomplex structures containing given compatible complex structure
on any open set containing the fixed-point set. 
Moreover, we determine the connected components of the fixed-point set arising 
from quaternionic $S^{1}$-actions on the quaternionic projective space.
Consequently, 
we generalize Pontecorvo's example 
$\mathrm{SO}^{\ast}(2n+2)/\mathrm{SO}^{\ast}(2n) \times \mathrm{SO}^{\ast}(2)$. 
\\


\end{abstract}

\tableofcontents

\section{Introduction}
\setcounter{equation}{0}

A quaternionic manifold $M$ is a smooth manifold equipped 
with a subbundle $Q$ of the endmorphism bundle of $TM$ 
which is locally generated by three almost complex structures satisfying 
the quaternionic relations and there exists a torsion-free affine connection 
preserving $Q$. Such a connection is called a quaternionic connection. 
Note that a quaternionic connection is not unique for $Q$. 
Several papers study compatible (almost) complex structures with a given $Q$ and  
quaternionic manifolds admitting such complex structures (see \cite{AMP, Bor} for example). 
The fundamental questions are 
\begin{enumerate}
  \item Are there compatible (globally defined) complex structures on $(M,Q)$?,  
  \item If there exists a compatible complex structure $I$,  
  is it possible to extend $I$ to a  (locally) hypercomplex structure $(I_{1},I_{2},I_{3})$ such that $I_{1}=I$? 
\end{enumerate}
In particular, Joyce \cite{J} introduces a quaternionic complex manifold, defined 
as a quaternionic manifold equipped with a nowhere vanishing twistor function. 
The twistor function is defined as a section lying in the kernel of the twistor operator. 
In this setting, there exists a unique quaternionic connection 
preserving both a distinguished compatible complex structure and a volume form.
Consequently, the holonomy group of this connection is contained in
$\mathrm{SL}(n,\mathbb{H}) \cdot \mathrm{U}(1)$, where $4n=\dim M$. 
This group appears in the Berger's list of possible holonomies of torsion-free connections 
(see \cite{Berger,MS}). Hitchin \cite{Hit} gives examples generated from 
quaternionic K\"ahler or hyperK\"ahler manifolds with an $S^{1}$-action.

Related to the problems as above, 
Pontecorvo \cite{P} shows that 
\[ \mathrm{SO}^{\ast}(2n+2)/\mathrm{SO}^{\ast}(2n) \times \mathrm{SO}^{\ast}(2) 
(\subset \mathbb{H}P^{n}) \]
is a manifold with the compatible complex structure which is not locally hypercomplex  
by using the twistor space. Note that there exists no almost complex structure on $\mathbb{H}P^{n}$. 
Hitchin \cite{Hit} points out that this manifold is given by the twistor(moment) function on $\mathbb{H}P^{n}$ 
with the $S^{1}$-action. 
The fixed-point set of this $S^{1}$-action on $\mathbb{H}P^{n}$ is $\mathbb{C}P^{n}$, which is contained in 
$\mathrm{SO}^{\ast}(2n+2)/\mathrm{SO}^{\ast}(2n) \times \mathrm{SO}^{\ast}(2)$. 
Hence, this motivates us to study fixed-point sets of a quaternionic $S^{1}$-action. 
In the quaternionic Feix-Kaledin construction \cite{BC}, 
the c-projective structure of the complex manifold,  
which arises as one of the connected components of the fixed-point set 
of the constructed quaternionic manifold under the $S^{1}$-action,  
plays an essential role. 
From this perspective, the study of the fixed-point sets of 
$S^{1}$-actions on quaternionic manifolds is also of interest.

We show that a connected component of the fixed-point set 
of a quaternionic $S^{1}$-action on a quaternionic manifold is a 
quaternionic or complex submanifold (Theorem \ref{thm_fixpts}). Note that a quaternionic submanifold
is totally geodesic with respect to any quaternionic connection and a complex submanifold which is the  
fixed-point set is totally geodesic with respect to a specific quaternionic connection. 
In particular, we prove that a connected component of the fixed-point set 
of a quaternionic $S^{1}$-action on $\mathbb{H}P^{n}$ 
is $\mathbb{H}P^{k}$ ($k \leq n$) or  
$\mathbb{C}P^{l}$ ($l \leq n$) (Theorem \ref{fix_hp_n}). 
The proof is done geometrically, and topologically, 
by \cite[Theorem 5.2 in Chapter VII]{Bre}, the fixed-point sets 
of a continuous $S^{1}$-actions on $\mathbb{H}P^{n}$
have the same integral 
cohomology as $\mathbb{H}P^{k}$ or as $\mathbb{C}P^{l}$.

In Riemannian submanifold geometry, the fixed-point set of an isometric action is a totally geodesic submanifold.
Although a conformal change of the ambient metric does not preserve the property of being totally geodesic, it does preserve the totally umbilic condition.
For this reason, totally umbilic submanifolds are the most fundamental and important objects in conformal submanifold geometry. 
The conformal geometry of $S^{4}$ can be viewed as the quaternionic geometry of $\mathbb{H}P^{1}$, and 
its higher-dimensional analogue is the quaternionic geometry of $\mathbb{H}P^{n}$.  
In quaternionic submanifold geometry (see \cite{Tsu}), 
it is shown that the $(2,0)+(0,2)$-part of the second fundamental form is independent of 
the choice of a quaternionic connection. 
Moreover, $\mathbb{C}P^{n}$ is characterized as 
a transversally complex submanifold in $\mathbb{H}P^{n}$ with vanishing 
$(2,0)+(0,2)$-part of the second fundamental form. 
Thus, $\mathbb{C}P^{n}$ plays a role analogous to that of
a totally umbilic submanifold in conformal geometry. 
From this perspective, it is also natural and important to study fixed-point sets of 
$S^{1}$-actions within the framework of quaternionic submanifold geometry.

We give detailed and elementary descriptions of twistor functions and quaternionic complex manifolds, reflecting their significant roles in this paper, in Sections \ref{sec_twist_fun} and \ref{sec_mu_conn}.
As observed in \cite{S,J}, in order to define a twistor operator which is independent of the choice of 
a quaternionic connection, one must consider an appropriate tensor product of 
$Q$ with a real line bundle of suitable weight.
We provide an explicit expression for this operator, and in particular, following \cite{AM2}, we describe the Ricci curvature of the distinguished quaternionic connection. 
Moreover, we express the first Chern class of a quaternionic complex manifold 
in terms of this Ricci curvature. 
We introduce the compatible complex structure of closed type on a quaternionic manifold 
(see Definition \ref{def_close_type}), which gives  
a collection of quaternionic complex manifolds and twistor functions on the quaternionic manifold (Proposition \ref{th_2a}). 
We obtain an equation involving the first Chern classes of the fixed-point set and of 
a quaternionically flat manifold with compatible complex structure of closed type (Theorem \ref{q_flat_fix}). 
This directly can be applied to the case of 
a quaternionic complex manifold. 
Indeed, in Section \ref{sec_hp_n}, we apply our results to $\mathbb{H}P^{n}$ 
to illustrate Pontecorvo's example
$\mathrm{SO}^{\ast}(2n+2)/\mathrm{SO}^{\ast}(2n) \times \mathrm{SO}^{\ast}(2)$ 
building on Hitchin's description in \cite{Hit2}. 
As shown in \cite[Theorem 3.9]{P}, there exists no locally hypercomplex structure containing 
the distinguished compatible complex structure. 
In the present paper, we prove that no open neighborhood of the fixed-point set 
admits a hypercomplex structure. 
Since this manifold is simply connected, no locally hypercomplex structure exists 
(see Example \ref{ex_qp1}).

We discuss the complex two-plane Grassmann manifold $Gr(2,n+1)$ in Section \ref{sec_weight_grass}. 
There is no compatible almost complex structure on $Gr(2,n+1)$. 
Nevertheless, there exists a compatible complex structure on 
$Gr(2,n+1) \setminus Gr(2,n)$. Indeed, this open manifold carries the quaternionic complex structure 
induced by a twistor function. However, this compatible complex structure cannot be extended 
to a hypercomplex structure (Example \ref{ex_gr1}).

\section{A fixed-point set of a quaternionic $S^{1}$-action}
\setcounter{equation}{0}

Throughout this paper, 
all manifolds are assumed to be smooth and without boundary 
and maps are assumed to be smooth 
unless otherwise mentioned. The space of sections of a vector bundle $E\rightarrow M$ 
is denoted by $\Gamma(E)$.

\subsection{A quaternionic manifold and its complex submanifold}

In dimensions greater than four, 
we say that $M$ is a {\cmssl quaternionic 
manifold} with {\cmssl quaternionic structure} $Q$ if 
$Q$ is a subbundle of $\mathrm{End}(T M)$ 
of rank $3$ which at every point $x\in M$ is spanned by endomorphisms  
$I_{1}$, $I_{2}$, $I_{3}\in \mathrm{End} (T_x M)$ satisfying
\begin{eqnarray}\label{quaternionic}
I_{1}^{2}=I_{2}^{2}=I_{3}^{2}=-\mathrm{id}, \,\, I_{1} I_{2}=-I_{2}I_{1}=I_{3}, 
\end{eqnarray}
and there exists a torsion-free connection $\nabla$ on $M$ 
such that $\nabla$ preserves $Q$, that is, 
$\nabla_X \Gamma(Q) \subset \Gamma(Q)$ for all $X\in\Gamma (TM)$. 
Such a torsion-free connection 
$\nabla$ is called a {\cmssl quaternionic connection} and 
the triplet $(I_{1}$, $I_{2}$, $I_{3})$ is called an 
{\cmssl admissible frame} of $Q$ at $x$. 
Note that we use the same letter $\nabla$ for the connection on ${\rm End}(T M)$ 
induced by $\nabla$. In four dimensions, we make the special definition that 
a quaternionic manifold is a self-dual conformal manifold.    
The dimension of the quaternionic manifold $M$ is denoted by $4n$. 
Note that a quaternionic connection is not unique, in fact, 
there is the following result \cite{Fu,AM1}. 

\begin{lemma}\label{q_conn_eq}
Let $\nabla^{1}$ and $\nabla^{2}$ be quaternionic connections on $(M,Q)$. Then there exists a 1-form $\xi$ on $M$ such that 
\begin{eqnarray}\label{q_conn}
\nabla^{2}_{X} Y = \nabla^{1}_{X} Y +S^{\xi}_{X} Y
\end{eqnarray}
for all $X$, $Y \in \Gamma(T M)$, where $S^{\xi}$
is defined by 
\begin{align*}
S^{\xi}_{X}Y =\; 
& \xi(X)Y + \xi(Y)X - \xi(I_{1}X)I_{1}Y - \xi(I_{1}Y)I_{1}X \\
   & - \xi(I_{2}X)I_{2}Y - \xi(I_{2}Y)I_{2}X - \xi(I_{3}X)I_{3}Y - \xi(I_{3}Y)I_{3}X. 
\end{align*}
Conversely, for a given quaternionic connection $\nabla^{1}$, the connection 
$\nabla^{2}$ given by the equation above is also a quaternionic connection. 
\end{lemma}

An {\cmssl almost hypercomplex manifold} is defined to be a manifold $M$ endowed 
with 3 almost complex structures 
$I_{1}$, $I_{2}$, $I_{3}$ satisfying the quaternionic relations (\ref{quaternionic}). 
If $I_{1}$, $I_{2}$, $I_{3}$ are integrable, then $M$ is called 
a {\cmssl hypercomplex manifold}. 
There exists a unique torsion-free connection on a hypercomplex manifold 
for which the hypercomplex structures are parallel. 
It is called the {\cmssl  Obata connection} \cite{O}. 
Obviously, hypercomplex manifolds are quaternionic manifolds with 
$Q=\langle I_{1},I_{2},I_{3} \rangle$. 
We say that $M$ is {\cmssl locally hypercomplex} if, 
for every point $x \in M$, 
there exists an open neighborhood $U$ of $x$ 
on which a hypercomplex structure $(I_{1},I_{2},I_{3})$ is defined.
Let $(\, \cdot \, , \, \cdot \,)$ be a metric on $Q$ defined by 
\[ ( A,B )=-\frac{1}{4n} \mathrm{Tr}AB \]
for $A,B \in Q$. 
We set $\| A \|=\sqrt{( A,A )}$ for $A \in Q$. 
Note that any admissible frame is orthonormal frame and 
any quaternionic connection is a metric connection with respect to this metric.   
For a  $(0,2)$-tensor $\theta$, we define the $(0,2)$-tensor $\Pi_{h}(\theta)$ by 
\[ \Pi_{h}(\theta)(X,Y)=\frac{1}{4} \left( \theta(X,Y)+\sum_{\alpha=1}^{3} 
\theta(I_{\alpha}X,I_{\alpha}Y) \right)
\]
for $X,Y \in TM$. We also define a tensor $B^{\nabla}$ on a quaternionic manifold  $(M,Q)$ by 
\begin{align*}
 B^{\nabla}=\frac{1}{4(n+1)} (Ric^{\nabla})^{a}+\frac{1}{4n}(Ric^{\nabla})^{s}
-\frac{1}{2n(n+2)} \Pi_{h} (Ric^{\nabla})^{s}
\end{align*}
for a quaternionic connection $\nabla$, where $Ric^{\nabla}$ is the Ricci tensor of $\nabla$ and 
$(Ric^{\nabla})^{s}$ (resp. $(Ric^{\nabla})^{a}$) is the symmetric (resp. skew-symmetric) 
part of $Ric^{\nabla}$. 
The {\cmssl Weyl tensor} 
\begin{align*}
 W:=R^{\nabla}-R^{B^{\nabla}}
\end{align*}
does not depend on a particular choice $\nabla$ 
(see \cite{AM1}), that is, 
it is invariant of $Q$. Here $R^{\nabla}$ is the curvature of $\nabla$ and $R^{B^{\nabla}}$ is defined by 
\begin{align*}
 R^{B^{\nabla}}_{X,Y}=S^{B^{\nabla}(Y, \, \cdot \,)}_{X}-S^{B^{\nabla}(X, \, \cdot \, )}_{Y}
\end{align*}
for $X$, $Y \in TM$.

A submanifold $N$ in a quaternionic manifold $(M,Q)$ is said to be an {\cmssl almost complex submanifold}
if there exists a section $I$ of $Q|_{N}$ such that 
\[ I^{2}=-\mathrm{Id} \,\,\, \mbox{and} \,\,\, I(TN) \subset TN. \]   
The almost complex structure on $N$ can be induced by $I$. In addition, if the 
induced almost complex structure is integrable, then $N$ is called a {\cmssl complex submanifold}. 
An almost complex submanifold $N$ is called a {\cmssl transversally complex submanifold} if 
\[ T_{x}N \cap L(T_{x}N) = \{ 0 \} \] 
at each point $x \in N$ for all element $L$ of orthogonal complement of 
$\langle I \rangle$ with respect to $(\, \cdot \, , \, \cdot \,)$.      
Note that the induced almost complex structure on a transversally complex submanifold 
is integrable. We refer to \cite{Tsu} for transversally complex submanifolds.

\subsection{A fixed-point set}

We start with the following lemma, which is well known 
in the case of an isometric action (see \cite{K}).

\begin{lemma}\label{fix_point}
Let $(M,\nabla)$ be a manifold with a torsion-free affine connection $\nabla$, and let  
$X$ be a vector field which generates one-dimensional subgroup $A$
of affine transformations of $(M,\nabla)$. 
We denote a connected component of the fixed-point set of $A$ by $F$. 
Then we have 
\begin{enumerate}
\item $T_{x}F = \mathrm{Ker}(\nabla X)_{x}$ for all $x \in F$, 
\item $F$ is totally geodesic in $(M,\nabla)$.  
\end{enumerate}
\end{lemma}

\begin{proof}
At first, we show (1). At $x \in F$, take $Y \in T_{x}F$ and a curve $\gamma:I \to F$ such that 
$\gamma(0)=x$ and $\gamma^{\prime}(0)=Y$. 
Since $X_{\gamma(t)}=0$ for all $t \in I$, we have 
\begin{align*} 
\nabla_{Y} X=\nabla_{\gamma^{\prime}(0)}X=
\left. \frac{d}{dt}
(P_{\gamma}|^{t}_{0})^{-1}(X_{\gamma(t)}) \right|_{t=0}=0,
\end{align*}
where $(P_{\gamma}|^{t}_{0}):T_{\gamma(0)}M \to T_{\gamma(t)}M$ is the parallel transformation along 
$\gamma$ with respect to $\nabla$. This shows that $T_{x}F \subset \mathrm{Ker}(\nabla X)_{x}$. 
Since $L_{X} \nabla=0$, we have $\nabla_{Y} (\nabla X)=-R^{\nabla}_{X,Y}$, hence it vanishes on $F$. 
Therefore $\dim \mathrm{Ker}(\nabla X)$ is constant on $F$, which means that 
$\cup_{x \in F} \mathrm{Ker}(\nabla X)_{x}$ is a smooth vector bundle over $F$. 
Take $Y \in \mathrm{Ker}(\nabla X)_{x}$ and extend it to a vector field 
$\bar{Y}$ on a neighborhood $U$ of $x$ such that $\bar{Y}|_{U \cap F} \in \mathrm{Ker}(\nabla X)$ 
and $\bar{Y}_{x}=Y$.  
To show that $Y \in T_{x}F$, it is sufficient to prove $\varphi_{s}(\exp(tY))=\exp(tY)$ for all $s$, 
where $\exp$ is the exponential map with respect to $\nabla$ and 
$\{ \varphi_{s} \}$ is the one parameter family of $X$
(Thus it holds that $\exp(tY) \in F$ and $Y \in T_{x}F$). 
Therefore, we shall show $\varphi_{s \ast}(Y)=Y$. 
This follows from 
\begin{align*}
[X,\bar{Y}]_{x}=(\nabla_{X} \bar{Y}-\nabla_{\bar{Y}}X)_{x}=0. 
\end{align*}
Here we have used $X_{x}=0$ and $\bar{Y}_{x}=Y \in \mathrm{Ker}(\nabla X)_{x}$ at $x \in F$. 
Next, we show (2). 
Since $X$ preserves $\nabla$, we have
\begin{align*}
0 =(L_{X} \nabla)_{Y}Z 
  =R^{\nabla}_{X,Y}Z - \nabla_{\nabla_{Y}Z}X 
\end{align*}
for $Y,Z \in \Gamma(TF)$. Hence it holds that $\nabla_{\nabla_{Y}Z}X=0$ on $F$, which means that 
$\nabla_{Y}Z \in \Gamma(\mathrm{Ker}(\nabla X))=\Gamma(TF)$. 
This shows that $F$ is totally geodesic. 
\end{proof}

Let $(M,Q)$ be a quaternionic manifold. 
We consider the normalizer 
\begin{align*}
N(Q)=\{ A \in \mathrm{End}(TM) \mid [A,I] \in Q\,\,\, \mbox{for all}\,\,\, I \in Q \} 
\end{align*}
and the centralizer 
\begin{align*}
Z(Q)=\{ A \in \mathrm{End}(TM) \mid [A,I]=0\,\,\, \mbox{for all}\,\,\, I \in Q \}. 
\end{align*}
Then we see that $N(Q)=Q+Z(Q)$. 
If $X \in \Gamma(TM)$ generates a quaternionic $S^{1}$-action, 
then $\nabla X$ is an element of $N(Q)$. Now we decompose 
\[ \nabla X=f_{Q}^{\nabla}+f_{Z}^{\nabla} \in \Gamma(Q)+\Gamma(Z(Q)). \]
On a fixed-point set $F$, $f_{Q}^{\nabla}$ and $f_{Z}^{\nabla}$ are 
independent of the choice of a quaternionic connection. Therefore, we can omit the symbol $\nabla$ for 
$f_{Q}^{\nabla}$ and $f_{Z}^{\nabla}$ on $F$, that is, we simply write $f_{Q}$ and $f_{Z}$ on $F$.  
Before proving Theorem \ref{thm_fixpts}, we establish several lemmas.

\begin{lemma}\label{cpx_submfd}
Let $(M,Q)$ be a quaternionic manifold endowed with an $S^{1}$-action preserving $Q$, 
and let $X$ denote the vector field generating the $S^{1}$-action. 
Let $\nabla$ be a quaternionic connection, and let $F$ be a connected component 
of the fixed-point set of the $S^{1}$-action which is not contained in 
$f_{Q}^{\nabla -1}(0)$. Then we have 
\begin{enumerate}
  \item $F$ is an almost complex submanifold of $(M \setminus f_{Q}^{\nabla -1}(0),Q)$,
  \item $I_{2}(TF) \subset \mathrm{Im} (\nabla X)$ for a local admissible frame 
  $(I_{1}=\frac{f_{Q}}{\| f_{Q} \|}, I_{2},I_{3})$ on $(M \setminus f_{Q}^{\nabla -1}(0),Q)$,
  \item $TF \cap I_{2}TF = TF \cap I_{3}TF = \{ 0 \}$ for such a local admissible frame.
\end{enumerate}
\end{lemma}

\begin{proof}
We first prove (1). 
Set $I:=\frac{f_{Q}}{\| f_{Q} \|}$ and decompose 
\[
\nabla X = f_{Q}+f_{Z}= \| f_{Q} \| I + f_{Z} \in Q+Z(Q)
\]
on $F$. For $Y \in \mathrm{Ker}(\nabla X)_{x}$, we compute
\begin{align*}
\| f_{Q} \| (\nabla_{IY} X)
&= \nabla_{\| f_{Q} \| IY} X 
= \nabla_{f_{Q}Y} X 
= (\nabla X)(f_{Q}Y) \\
&= (f_{Q}+f_{Z})(f_{Q}Y) 
= f_{Q}(f_{Q}+f_{Z})(Y) 
= f_{Q}(\nabla X)(Y)=0.
\end{align*}
Hence $T_{x}F=\mathrm{Ker}(\nabla X)_{x}$ is $I$-invariant, 
which proves (1). 
Next we prove (2). 
Take $Y \in \mathrm{Ker}(\nabla X)$. Then 
\[
0=\nabla_{Y} X=f_{Q}Y+f_{Z}Y,
\]
and hence
\[
I_{1}Y=\frac{1}{\| f_{Q} \|}f_{Q}Y
=-\frac{1}{\| f_{Q}\|} f_{Z}Y.
\]
Since
\[
\frac{1}{2 \| f_{Q} \|}(f_{Q}-f_{Z})(Y)
=-\frac{1}{ \| f_{Q} \|} f_{Z}Y,
\]
we obtain
\[
I_{1}Y=\frac{1}{2 \| f_{Q} \|}(f_{Q} - f_{Z})(Y).
\]
Therefore,
\begin{align}\label{eq1}
I_{2}I_{1}Y
&=\frac{1}{2 \| f_{Q} \|}(-f_{Q} - f_{Z})(I_{2}Y) 
=- \frac{1}{2 \| f_{Q} \|}(f_{Q}+ f_{Z})(I_{2}Y) \\
&=- \frac{1}{2 \| f_{Q} \|} (\nabla X)(I_{2}Y)
\in \mathrm{Im} (\nabla X). \nonumber
\end{align}
Since $\mathrm{Ker}(\nabla X)$ is $I_{1}$-invariant, we obtain 
\[
I_{2}(TF)=I_{2}(\mathrm{Ker}(\nabla X)) 
\subset \mathrm{Im} (\nabla X).
\]
Finally, we prove (3). 
Set $V:=TF \cap I_{2}TF = TF \cap I_{3}TF$. 
Then $V$ is $Q$-invariant. 
Let $Y \in V$. Then $Y$ and $I_{\alpha}Y$ belong to $V$. 
Substituting $I_{3}Y \in V \subset \mathrm{Ker}(\nabla X)$ 
into (\ref{eq1}) in place of $Y$, we obtain
\[
Y
=- \frac{1}{2 \| f_{Q} \|} (\nabla X)(I_{1}Y).
\]
Since $I_{1}Y \in TF=\mathrm{Ker}(\nabla X)$, 
it follows that $Y=0$, which proves (3).
\end{proof}

\begin{lemma}\label{der_nab_q}
Let $\nabla$ be a quaternionic connection, and let 
$X$ be a vector field preserving $Q$. Then we have 
\begin{align*}
\nabla_{Y} f_{Q}^{\nabla}
&=
-\frac{1}{4n} \sum_{\alpha=1}^{3} (\mathrm{Tr} H^{\nabla}_{Y, I_{\alpha}(\, \cdot \,)}X      ) I_{\alpha}
\end{align*}
for $Y \in TM$, where $H^{\nabla}$ is the Hessian with respect to $\nabla$. In addition, 
if $L_{X} \nabla=0$, then we have
\begin{align*}
\nabla  f_{Q}^{\nabla}
&=
\frac{1}{4n} \sum_{\alpha=1}^{3} (\mathrm{Tr} R^{\nabla}_{X,(\, \cdot \, )} \circ I_{\alpha}) I_{\alpha}.
\end{align*}
\end{lemma}

\begin{proof}
Since an admissible frame $(I_{1},I_{2},I_{3})$ is an orthonormal frame with respect to 
$(\, \cdot \, , \, \cdot \,)$, 
we have
\begin{align*}
\nabla_{Y} f_{Q}^{\nabla}
=&-\frac{1}{4n} \sum_{\alpha=1}^{3} Y(\mathrm{Tr} (\nabla X)I_{\alpha}) I_{\alpha}
   -\frac{1}{4n} \sum_{\alpha=1}^{3} (\mathrm{Tr} (\nabla X)I_{\alpha}) (\nabla_{Y}I_{\alpha})\\
=&-\frac{1}{4n} \sum_{\alpha=1}^{3} (\mathrm{Tr} (\nabla_{Y} (\nabla X))I_{\alpha}) I_{\alpha}
   -\frac{1}{4n} \sum_{\alpha=1}^{3} (\mathrm{Tr} (\nabla X)(\nabla_{Y} I_{\alpha})) I_{\alpha} \\
  &-\frac{1}{4n} \sum_{\alpha=1}^{3} (\mathrm{Tr} (\nabla X)I_{\alpha}) (\nabla_{Y}I_{\alpha}). 
\end{align*}
The second term and third term are canceled since $\nabla$ is a quaternionic connection. 
By $\nabla_{Y} (\nabla X)=H^{\nabla}_{Y, (\, \cdot \,)}X$, we have 
\begin{align*}
\nabla_{Y} f_{Q}^{\nabla}
=-\frac{1}{4n} \sum_{\alpha=1}^{3} (\mathrm{Tr} H^{\nabla}_{Y, I_{\alpha}(\, \cdot \,)}X      ) I_{\alpha}. 
\end{align*}
If $L_{X}\nabla=0$, then we see $R^{\nabla}_{X,Y}Z=-H^{\nabla}_{Y,Z}X$ for all $Y,Z \in TM$. 
\end{proof}

We can now state and prove the following theorem.

\begin{theorem}\label{thm_fixpts}
Let $(M,Q)$ be a quaternionic manifold endowed with an $S^{1}$-action preserving $Q$. 
If $F$ is a connected component of the fixed-point set of the $S^{1}$-action, 
then $\| f_{Q} \|$ is constant on $F$. 
If $f_{Q}=0$ on $F$, then $F$ is a quaternionic submanifold. Otherwise,  
$F$ is a transversally complex and totally geodesic submanifold with respect to 
a quaternionic connection and its dimension satisfies 
\begin{align}\label{2_5_dim_ineq}
\dim F \leq  \frac{1}{2}\dim M. 
\end{align}
\end{theorem}

\begin{proof}
Let $X$ be a vector field on $M$ generating the $S^{1}$-action. 
By \cite[Corollary 4.3]{CH}, there exists a quaternionic connection $\nabla$ such that $L_{X}\nabla=0$. 
Since 
\[ H^{\nabla}_{Y,Z}X=-R^{\nabla}_{X,Y}Z=0 \]
for all $Y,Z \in \Gamma(TF)$,  
Lemma \ref{der_nab_q} implies that $\nabla f_{Q}=0$, and hence $\| f_{Q} \|$ is constant on $F$. 
If $ f_{Q} =0$ on $F$, then $\nabla X=f_{Z} \in Z(Q)$ along $F$. 
It follows that $TF$ is $Q$-invariant, and consequently $F$ is a quaternionic submanifold.  
If $f_{Q} \neq  0$ on $F$, then Proposition \ref{cpx_submfd} shows that
$F$ is an almost complex and totally geodesic submanifold. 
Moreover, the induced almost complex structure $I_{1}$ of $F$ is integrable since  
\[ (\nabla_{Y} I_{1})(Z)=0 \]
for all $Y,Z \in TF$. By (3) of Lemma \ref{cpx_submfd}, it follows that $F$ is a transversally complex submanifold. 
The inequality \eqref{2_5_dim_ineq} follows from (3) of Lemma \ref{cpx_submfd}. 
\end{proof}

\begin{remark}
{\rm 
The highest dimensional case for $f_{Q} \neq 0$ is proved in the first part of \cite[Theorem 4]{BC}. 
We refer to \cite{B} for the case of a quaternionic K\"ahler manifold. 
}
\end{remark}

\begin{remark}
{\rm 
Any quaternionic submanifold in a quaternionic manifold 
is totally geodesic with respect to any quaternionic connection (see Lemma \ref{q_inv}). 
On the other hand, 
the latter is that $F$ is totally geodesic with respect to a specific quaternionic connection $\nabla$.  
}
\end{remark}

\begin{remark}\label{rem_para}
{\rm 
Consider the case of $f_{Q} \neq 0$ and a vector subbundle 
$V:=TF \oplus I_{2}TF(=TF \oplus I_{3}TF)$ of $TM|_{F}$ over $F$. 
Then $V$ is $Q$-invariant and parallel subbundle with respect to 
$\nabla$ such that $F$ is totally geodesic. 
By Lemma \ref{q_conn_eq}, the subbundle $V$ is parallel with respect to
any quaternionic connection of $M$. In particular, we have 
\[ \nabla_{X}Y \in \Gamma(V) \]
for any $X$, $Y \in \Gamma(TF)$ with respect to any quaternionic connection $\nabla$. 
This allows us to project $R^{\nabla}_{X,Y}Z$ ($X$, $Y$, $Z \in TF$) 
onto its tangent component using the decomposition of $V$. 
}
\end{remark}

We give examples which are discussed later.

\begin{example}\label{cpx_pr}
{\rm 
Consider the $S^{1}$-action on the quaternionic projective space $\mathbb{H}P^{n}$ given by 
$e^{i \theta}[z_{1}:\cdots:z_{n+1}]:=[e^{i \theta} z_{1}:\cdots: e^{i \theta} z_{n+1}]$. 
Then $F=\mathbb{C}P^{n}$.
}
\end{example}

\begin{example}\label{cpx_gr}
{\rm 
Let $Gr(2,m)$ denote the Grassmannian of all $2$-dimensional subspaces in
$\mathbb{C}^{m}$. Consider the $S^{1}$-action on $Gr(2,m)$ defined by 
\[ e^{i \theta} W := \mathrm{Span} 
\{ (e^{i p \theta} u_{1}, e^{i q \theta} u_{2}, \dots , e^{i q \theta} u_{m}), 
(e^{i p \theta} v_{1}, e^{i q \theta} v_{2}, \dots , e^{i q \theta} v_{m}) \}
\]
for each subspace $W=\mathrm{Span}\{ u,v \}$, where
$u=(u_{1}, u_{2}, \dots, u_{m})$, $v=(v_{1}, v_{2}, \dots, v_{m}) \in \mathbb{C}^{m}$, 
and $p$, $q \in \mathbb{N}$ are distinct with $gcm(p,q)=1$. 
Under this action, $F=\mathbb{C}P^{m-2}$ and $Gr(2,m-1)$.
Here we identify $Gr(2,m-1)$ with the space of all $2$-dimensional subspaces of the form
$\mathrm{Span} \{  (0, u_{2}, \dots, u_{m}), (0, v_{2}, \dots, v_{m}) \}$. 
}
\end{example}

\section{The twistor function and a quaternionic complex manifold}
\setcounter{equation}{0}

\subsection{The twistor operator}\label{sec_twist_fun}
Since the twistor operator has a significant role in this paper, 
we give a detailed and elementary description for it. 
Let $M$ be an $m$-dimensional manifold, and set $L:=\Lambda^{m}(T^{\ast}M)$. 
Take coordinate systems 
$\{ U_{\alpha}, (x^{\alpha}_{1},\dots,x^{\alpha}_{m}) \}_{\alpha \in \Lambda}$ of $M$. 
Consider a collection of maps $t_{\alpha \beta} : U_{\alpha} \cap U_{\beta} \to \mathbb{R}^{\ast}$ 
defined by
\begin{align*}
U_{\alpha} \cap U_{\beta} \ni x \mapsto \left| \det \left(
\frac{\partial x^{\beta}_{i}}{\partial x^{\alpha}_{j}} 
\right)_{x} \right|^{-s}
\end{align*}
for  $s \in \mathbb{R}$. 
Since the collection $\{ t_{\alpha \beta }\}$ satisfies the cocycle condition, 
we can obtain the vector bundle $L^{s}$ of $\mathrm{rank} L^{s}=1$.
We have $L^{s} \otimes L^{s^{\prime}} \cong L^{s+s^{\prime}}$. 
In particular, it holds that 
\[ \underbrace{(L^{\frac{q}{p}}) \otimes \cdots \otimes (L^{\frac{q}{p}})}_{p \,\, times}
\cong \underbrace{L \otimes \cdots \otimes L}_{q \,\, times}\] 
for $q/p \in \mathbb{Q}^{\ast}$. 
Indeed, a transition function of a tensor product of  
two line bundles over $M$ is given by $\{ t_{\alpha \beta} t^{\prime}_{\alpha \beta} \}$, 
where $\{ t_{\alpha \beta} \}$, $\{ t^{\prime}_{\alpha \beta} \}$ are transition functions of 
the bundles respectively. 
Let $\nabla$ be a torsion-free affine connection on $M$. 
We denote the connection on $L$ induced from $\nabla$ by the same symbol $\nabla$ 
and define the connection on $L^{s}$ which is denoted by $\nabla^{s}$ 
($\nabla^{1}=\nabla$). 
The connection $\nabla^{\frac{q}{p}}$ ($q/p \in \mathbb{Q}^{\ast}$) is also characterized by 
\begin{align*}
\nabla^{q} (f \otimes \cdots \otimes f)=\sum_{i=1}^{p}
e\otimes \cdots \otimes( \underbrace{\nabla^{\frac{q}{p}} e}_{i-th}) \otimes \cdots \otimes e
\end{align*}
for $\underbrace{e\otimes \cdots \otimes e}_{p \,\, times}
=\underbrace{f \otimes \cdots \otimes f}_{q \,\, times} $. In fact, we have
\begin{align*}
\sum_{i=1}^{p}
e\otimes \cdots \otimes( \underbrace{\nabla^{\frac{q}{p}}_{X} e}_{i-th}) \otimes \cdots \otimes e
&=\frac{q}{p} \, \theta_{L}(X) (p \, (e \otimes \cdots \otimes e)) \\
&=q \, \theta_{L}(X) (f \otimes \cdots \otimes f)=\nabla^{q} (f \otimes \cdots \otimes f),
\end{align*}
where $\theta_{L}$ is defined by $\nabla f=\theta_{L} \otimes f$ 
for a nowhere vanishing section $f \in \Gamma(L)$. 
\begin{lemma}
Let $f$ be a nowhere vanishing section of $L$ and $\nabla$ be a connection such that 
$\nabla f=0$. If $e \in \Gamma(L^{\frac{q}{p}})$ $with 
\underbrace{e\otimes \cdots \otimes e}_{p \,\, times}
=\underbrace{f \otimes \cdots \otimes f}_{q \,\, times}$, then we have $\nabla^{\frac{q}{p}} e=0$. 

\end{lemma}

Let $(M,Q)$ be a quaternionic manifold, and let $\nabla$, $\tilde{\nabla}$ be quaternionic connections.

\begin{lemma}\label{q_conn_eq2}
If $\tilde{\nabla}= \nabla+S^{\xi}$, 
the difference of $\tilde{\nabla}^{s}$ and $\nabla^{s}$ as connections on $L^{s}$ is given by 
$\tilde{\nabla}^{s} - \nabla^{s}=-4s(n+1) \xi \otimes \mathrm{Id}_{L^{s}}$. 
\end{lemma}

\begin{proof}
Let $(\xi_{1},\dots,\xi_{4n})$ be a tangent frame of $M$. 
The connection forms with respect to $(\xi_{1},\dots,\xi_{4n})$ for 
quaternionic connections $\tilde{\nabla}$ and $\nabla$ are denoted by 
$\tilde{\omega}^{j}_{i}$ and $\omega^{j}_{i}$. 
Then it holds that 
\begin{align*}
\sum_{i=1}^{4n}  \tilde{\omega}^{i}_{i}(X)=\sum_{i=1}^{4n} {\omega}^{i}_{i}(X) + \mathrm{Tr}S^{\xi}_{X}. 
\end{align*}
The connection forms $\tilde{\theta}_{L}$ and $\theta_{L}$
of $\tilde{\nabla}$ and $\nabla$ of $L$ with respect to 
$f=\xi_{1}^{\ast} \wedge \dots \wedge \xi_{4n}^{\ast}$ are given by 
\begin{align*}
\tilde{\theta}_{L}=-\sum_{i=1}^{4n}  \tilde{\omega}^{i}_{i}, \,\,\ 
\theta_{L}=-\sum_{i=1}^{4n} {\omega}^{i}_{i},  
\end{align*}
where $(\xi_{1}^{\ast},\dots, \xi_{4n}^{\ast})$ is the dual frame of $(\xi_{1},\dots,\xi_{4n})$. 
From $\mathrm{Tr}S^{\xi}_{X}=4(n+1)\xi(X)$ (see \cite{AM1}), it holds that 
\begin{align*}
\tilde{\nabla}^{s}_{X} e - \nabla^{s}_{X} e = s(\tilde{\theta}_{L}(X)-\theta_{L}(X)) e  
=-s(\mathrm{Tr}S^{\xi}_{X}) e =-4s(n+1)\xi(X) e  
\end{align*}
for $X \in TM$ and $e \in \Gamma(L^{s})$. 
\end{proof}

The connection on $Q \otimes L^{s}$ is denoted by $\nabla^{Q, s}$, which is 
induced by $\nabla$ and $\nabla^{s}$.   
We take a local section $I$ of $Q$ 
such that $(I_{1}:=I,I_{2},I_{3})$ is an admissible frame of $Q$ 
and $e \in \Gamma(L^{s})$. 
By $\tilde{\nabla}_{X} I-\nabla_{X}I=[S^{\xi}_{X},I]$ and Lemma \ref{q_conn_eq2}, 
we have 
\begin{align*}
\tilde{\nabla}^{Q, s}_{X} (I \otimes e) 
&= (\tilde{\nabla}_{X} I ) \otimes e + I \otimes (\tilde{\nabla}^{s}_{X} e) \\
&= (\nabla_{X} I ) \otimes e +[S^{\xi}_{X},I] \otimes e 
+I \otimes (D^{s}_{X} e) -4s(n+1) \xi(X) (I \otimes e)\\
&={\nabla}^{Q, s}_{X} (I \otimes e)+[S^{\xi}_{X},I] \otimes e - 4s(n+1) \xi(X) (I \otimes e). 
\end{align*}
Since $[S^{\xi}_{X},I]=[S^{\xi}_{X},I_{1}]=2 \xi(I_{2}X)I_{3}-2 \xi(I_{3}X)I_{2}$, we see that 
\begin{align*}
  &\tilde{\nabla}^{Q, s}_{X} (I \otimes e) - {\nabla}^{Q, s}_{X} (I \otimes e) \\
=&[S^{\xi}_{X},I] \otimes e - 4s(n+1) \xi(X) (I \otimes e) \\
=&\{ -2(\xi \circ I_{1})(I_{3}X) I_{3}-2 (\xi \circ I_{1})(I_{2}X) I_{2}
   +4s(n+1) (\xi \circ I_{1})(I_{1} X) I_{1} \} \otimes e.  
\end{align*}
Consider 
\begin{align*}
T^{\ast}M_{(1)}:=\left\{ A^{\xi}:= \sum_{\alpha=1}^{3} (\xi \circ I_{\alpha}) \otimes I_{\alpha}
\mid \xi \in T^{\ast}M \right\}. 
\end{align*}
If $4s(n+1)=-2$, then 
\begin{align*}
\tilde{\nabla}^{Q, s} (I \otimes e) - {\nabla}^{Q, s}(I \otimes e) = -2 A^{\xi \circ I}(X) \otimes e 
(\in T^{\ast}M_{(1)} \otimes L^{s}). 
\end{align*}
By \cite[Lemma 1.2 (2)]{AM1}, the decomposition 
\begin{align*}
T^{\ast}M \otimes Q=T^{\ast}M_{(1)} \oplus (T^{\ast}M \otimes Q)_{0}
\end{align*}
holds, where $(Q \otimes T^{\ast}M)_{0}$ is the subspace of tensors with all zero contractions. 
Therefore, it holds that 
\begin{align*}
T^{\ast}M \otimes Q \otimes L^{s} 
= (T^{\ast}M_{(1)} \otimes L^{s}) \oplus ((T^{\ast}M \otimes Q)_{0} \otimes L^{s}). 
\end{align*}
By viewing 
a connection on $Q \otimes L^{s}$ as 
a map from $\Gamma(Q \otimes L^{s})$ to 
$\Gamma(T^{\ast}M \otimes Q \otimes L^{s})$, we have the following proposition.

\begin{proposition}[\cite{S}]
If $s=-\frac{1}{2(n+1)}$, then we have 
\begin{align*}
p \circ \tilde{\nabla}^{Q, s} = p \circ \nabla^{Q, s},
\end{align*}
where $p: T^{\ast}M \otimes Q \otimes L^{s} \to (T^{\ast}M \otimes Q)_{0} \otimes L^{s}$ 
is the projection onto the second factor.  
\end{proposition}

Hereafter we set 
\[ s_{0}:=-\frac{1}{2(n+1)}. \] 
We can define a operator $D^{s_{0}}:=p \circ \nabla^{Q, s_{0}}  : \Gamma(Q \otimes L^{s}) \to 
\Gamma((T^{\ast}M \otimes Q)_{0} \otimes L^{s_{0}})$, which is independent of the choice 
of a quaternionic connection. 
We call $D^{s_{0}}$ the {\cmssl twistor operator}. 
A section of $Q \otimes L^{s_{0}}$ in $\mathrm{Ker} \, D^{s_{0}}$ 
is called a {\cmssl twistor function} (\cite{J}). Clearly, $\mu$ is a twistor function if and only if 
there exists $\xi \in T^{\ast}M$ such that 
\begin{align}\label{tw_eq}
{\nabla}^{Q, s_{0}} \mu 
=A^{\xi} \otimes e = \sum_{\alpha=1}^{3} (\xi \circ I_{\alpha}) \otimes I_{\alpha} \otimes e. 
\end{align}
If we choose a quaternionic connection $\nabla$ such that 
$\nabla \nu=0$ for some volume element $\nu$ on $M$, 
the section 
$e \in \Gamma(L^{s_{0}})$ with 
$\underbrace{e\otimes \cdots \otimes e}_{2(n+1) \,\, times}=\nu^{\ast} \in \Gamma(L^{-1})(=\Gamma(L^{\ast}))$ 
satisfies $\nabla^{s_{0}} e=0$, where $\nu^{\ast}$ is the dual frame of $\nu$. 
Therefore,  
the equation (\ref{tw_eq}) is equivalent to 
\begin{align}\label{tw_eq1}
\nabla \bar{\mu}=\sum_{\alpha=1}^{3} (\xi \circ I_{\alpha}) \otimes I_{\alpha}, 
\end{align}
where $\mu=\bar{\mu} \otimes e$. 
We say that $\xi$ is called a {\cmssl corresponding one-form} of $\bar{\mu}$ (with respect to $\nabla$). 
Note that 
for a given volume form $\nu$
there exists a unique quaternionic connection $\nabla$ such that 
$\nabla \nu =0$. 
Hereafter for a section $e \in \Gamma(L^{s_{0}})$ satisfying 
\[ \underbrace{e\otimes \cdots \otimes e}_{2(n+1) \,\, times}=\nu^{\ast} \]
for a volume form $\nu$, the pair $(e,\nu)$ is called an 
{\cmssl associated pair}. Any section $e \in \Gamma(L^{s_{0}})$ defines a volume element $\nu$ 
such that $(e,\nu)$ is an associated pair. 

\begin{lemma}
For any volume form $\nu$, there exists $e \in \Gamma(L^{s_{0}})$ such that 
either $(e,\nu)$ or $(e, -\nu)$ is an associated pair.  
\end{lemma}

\subsection{The Swann bundle}\label{sec_swann}
We recall the Swann bundle, since it is convenient to lift a twistor function to the Swann bundle. 
See \cite{S, PPS, CH} for the Swann bundle. 
Here we adapt the descriptions in \cite{CH} for the Swann bundle as follows.

Let $S=S(M)$ be the principal ${\rm SO(3)}$-bundle of admissible frames 
 $(I_{1}$, $I_{2}$, $I_{3})$ over a quaternionic manifold $(M,Q)$. 
The bundle projection of $S$ is denoted by $\pi_{S}$. 
We take a basis $(e_{1},e_{2},e_{3})$ of 
${\mathbb R}^{3} \cong {\rm Im}\, {\mathbb H} \cong \mathfrak{sp}(1)
\cong \mathfrak{so}(3)$ so that  
\[ [e_{\alpha},e_{\beta}]=2e_{\gamma}  \] 
for any cyclic permutation $(\alpha,\beta,\gamma)$. 
Hereafter $(\alpha,\beta,\gamma)$ 
will be always a cyclic permutation, whenever 
the three letters appear in an expression. 
The fundamental vector field generated by $e_{\alpha}$ is denoted by 
$\widetilde{e}_{\alpha}$ ($\alpha=1,2,3$).
The quaternionic connection induces a principal connection 
$\theta:TS \to \mathfrak{so}(3)$. 
Moreover, we consider the principal ${\mathbb R}^{>0}$-bundle 
$S_{0}:=(\Lambda^{4n}(T^{\ast}M) \setminus \{ 0 \})/\{ \pm 1 \}$ over $M$. 
The bundle projection of $S_{0}$ is denoted by $\pi_{S_{0}}$. 
A quaternionic connection induces 
also a principal connection
$\theta_{0}:TS_{0} \to {\mathbb R} = \mathrm{Lie}\, (\mathbb{R}^{>0})$. 
The product $S_{0} \times S$ is a principal 
${\mathbb R}^{>0} \times {\rm SO}(3)(={\mathbb H}^{\ast}/\{ \pm 1\})$-bundle over 
$M \times M$. 
The ${\mathbb R}^{4} ( \cong {\mathbb R} \oplus \mathfrak{so}(3))$-valued 
one-form $(\theta_{0} \circ pr_{TS_{0}} ,  \theta \circ pr_{TS})
=(\theta_{0} \circ pr_{TS_{0}}, 
\theta_{1}\circ pr_{TS},\theta_{2}\circ pr_{TS},\theta_{3}\circ pr_{TS})$ 
is a principal connection on $S_{0} \times S$, where 
$pr_{TS_{0}}$ (resp. $pr_{TS}$) is the projection from 
$T(S_{0} \times S) \cong TS_{0} \times TS$ onto $TS_{0}$ (resp. $TS$). 
Let $\triangle: M \to M \times M$ be the diagonal map defined by 
$\triangle(x)=(x,x)$ for each $x \in M$. 
The pullback bundle 
\[ \hat{M}:=\triangle^{\ast} (S_{0} \times S)
= \{ (x,(\rho , s))\in M \times (S_0 \times S) \mid x= \pi_{S_0}(\rho ) = \pi_S (s)\}\] 
is a principal 
${\mathbb R}^{>0} \times {\rm SO}(3)$-bundle over 
$M$ and $\bar{\theta}:=\triangle_{\#}^{\ast}
(\theta_{0}\circ pr_{TS_{0}},\theta \circ pr_{TS})$
is a principal connection on $\hat{M}$, 
where $\triangle_{\#} : \hat{M} \to S_{0} \times S$ 
is the canonical bundle map. 
The bundle projection of $\hat{M}$ onto $M$ is denoted by $\hat{\pi}$.

Set $e_{0}:=1 \in  {\mathbb R}\, (\cong T_{1} {\mathbb R}^{>0} )$ 
and $Z_{0}^{c}:=c \, \widetilde{e}_{0}$ 
for a nonzero real number $c$.
We denote the horizontal lifts relative to the connections 
$\bar{\theta}$ by $(\,\,\,\,)^{\bar{h}}$. 
An almost hypercomplex structure 
$(\hat{I}^{\bar{\theta}, c}_{1},\hat{I}^{\bar{\theta}, c}_{2},
\hat{I}^{\bar{\theta}, c}_{3})$ 
on $\hat{M}$ is defined by 
\[ \hat{I}^{\bar{\theta}, c}_{\alpha}Z^{c}_{0}=- Z_{\alpha},\quad  
\hat{I}^{\bar{\theta}, c}_{\alpha}Z_{\alpha}= Z_{0}^{c}, \quad  
\hat{I}^{\bar{\theta}, c}_{\alpha}Z_{\beta}=Z_{\gamma}, \quad  
\hat{I}^{\bar{\theta}, c}_{\alpha}Z_{\gamma}=-Z_{\beta} \]
and
\[ (\hat{I}^{\bar{\theta}, c}_{\alpha})_{(x,(\rho,s))} (X)
=(I_{\alpha} (\hat{\pi}_{\ast} X))^{\bar{h}}_{(x,(\rho,s))} \]
for all horizontal vector $X$ at $(x,(\rho,s)) \in \hat{M}$, 
where $Z_{\alpha}=\widetilde{e}_{\alpha}$ and $s=(I_{1},I_{2},I_{3})$. 
Note that the triple 
$(\hat{I}^{\bar{\theta}, c}_{1},\hat{I}^{\bar{\theta}, c}_{2},\hat{I}^{\bar{\theta}, c}_{3})$
depends on the connection form $\bar{\theta}$
and $c$ in general. If $c=-4(n+1)$, then 
the almost hypercomplex structure is integrable 
and independent of $\nabla$. 
When $c \neq - 4(n+1)$, 
the almost hypercomplex structure is integrable  
if and only if $(Ric^{\nabla})^{a}$ is $Q$-hermitian (\cite{CH}). 
The bundle $\hat{M}$ over $M$ is called the {\cmssl Swann bundle}. 
We consider the lifted map of a twistor function on $M$ to the Swann bundle $\hat{M}$.  

\begin{lemma}\label{qh_moment}
Let $(M,Q)$ be a quaternionic manifold, and let $\hat{M}$ be the Swann bundle of $M$. 
For  a twistor function $\mu=\bar{\mu} \otimes e$ of $(M,Q)$, 
we define a function $\hat{\mu}:\hat{M} \to \mathbb{R}^{3}$ 
by
\begin{align*}
\hat{\mu}((r,(I_{2},I_{2},I_{3})))
&:=(\hat{\mu}_{1}(r,(I_{2},I_{2},I_{3})),\hat{\mu}_{2}(r,(I_{2},I_{2},I_{3})), 
\hat{\mu}_{3}(r,(I_{2},I_{2},I_{3}))) \\
&:=(r^{\frac{2}{c}} ( \bar{\mu},I_{1} ), 
r^{\frac{2}{c}}( \bar{\mu},I_{2} ), r^{\frac{2}{c}} (\bar{\mu},I_{3} ))
\end{align*}
for $(r,(I_{2},I_{2},I_{3})) \in \hat{M}$. 
If we consider a quaternionic connection $\nabla$ such that $\nabla (e \otimes \cdots \otimes e)=0$ 
or $c=-4(n+1)$, then $\hat{\mu}$ satisfies the CR-condition, that is, 
\begin{align*}
d \hat{\mu}_{1} \circ \hat{I}_{1}^{\bar{\theta},c}=d \hat{\mu}_{2} \circ \hat{I}_{2}^{\bar{\theta},c}=
d \hat{\mu}_{3} \circ \hat{I}_{3}^{\bar{\theta},c}.
\end{align*}
In particular, suppose that a Lie group $G$ acts freely on $M$ 
and preserves the quaternionic structure $Q$, 
and that $\mu$ is a quaternionic moment map. 
Then $\hat{\mu}$ 
is a hypercomplex moment map for the induced action of $G$ 
on the Swann bundle.
\end{lemma}

\begin{proof}
We have 
\begin{align*}
d \hat{\mu} (Z_{0}^{c}) 
&=(c r \frac{\partial}{\partial r}) \hat{\mu}=2 \hat{\mu}, \,\,\, \mbox{that is,} \,\,\,
 d \hat{\mu}_{\alpha} (Z_{0})=2 \hat{\mu}_{\alpha}. 
\end{align*}
Take a base $(e_{1},e_{2},e_{3})$ of $\mathfrak{so}(3)$ as 
\begin{align*}
e_{1}=
\begin{pmatrix}
   0 & 0 & 0 \\
   0 & 0 & -2 \\
   0 & 2 & 0
\end{pmatrix}, 
e_{2}=
\begin{pmatrix}
   0 & 0 & 2 \\
   0 & 0 & 0 \\
   -2 & 0 & 0
\end{pmatrix},
e_{3}=
\begin{pmatrix}
   0 & -2 & 0 \\
   2 & 0 & 0 \\
   0 & 0 & 0
\end{pmatrix}. 
\end{align*}
Since 
\begin{align*}
\exp(te_{1})=
\begin{pmatrix}
   1 & 0 & 0 \\
   0 & \cos 2t & -\sin 2t \\
   0 & \sin 2t & \cos 2t
\end{pmatrix}, 
\end{align*}
we have 
\begin{align*}
d \hat{\mu}(Z_{1})
=\left. \frac{d}{dt} \hat{\mu}(((I_{1},I_{2},I_{3}),r)\exp(te_{1})) \right|_{t=0}
=(0, 2\hat{\mu}_{3}, -2 \hat{\mu}_{2} )
\end{align*}
at $(r,(I_{1},I_{2},I_{3})) \in \hat{M}$. Similarly, it holds that 
\begin{align*}
d \hat{\mu}(Z_{2}) =(-2\hat{\mu}_{3}, 0, 2 \hat{\mu}_{2} ), 
d \hat{\mu}(Z_{3}) =(2\hat{\mu}_{2}, -2 \hat{\mu}_{1},0 ).  
\end{align*}
Using these equations, we can check 
\begin{align*}
d \hat{\mu}_{1} \circ \hat{I}_{1}^{\bar{\theta},c}=d \hat{\mu}_{2} \circ \hat{I}_{2}^{\bar{\theta},c}=
d \hat{\mu}_{3} \circ \hat{I}_{3}^{\bar{\theta},c}
\end{align*}
on vertical vectors. 
Next we take a tangent vector $Y$ at $x=\hat{\pi}((r,(I_{1},I_{2},I_{3}))) \in M$ and 
a curve $\gamma$ on $M$ such that $\gamma(0)=x$ and $\gamma^{\prime}(0)=Y$. 
If we consider a horizontal lift $\tilde{\gamma}$ of $\gamma$ starting from $(r,(I_{1},I_{2},I_{3}))$, 
then $\tilde{\gamma}(t)=(\bar{r}, (\bar{I}_{1}, \bar{I}_{2}, \bar{I}_{3}))_{\gamma(t)}$ gives 
a $\nabla$-parallel frame $(\bar{I}_{1}, \bar{I}_{2}, \bar{I}_{3})$ along $\gamma$ and we see that $\bar{r}$ is constant along $\tilde{\gamma}$. Then we have 
\begin{align*}
d \hat{\mu}(\tilde{Y}) &=(r^{\frac{2}{c}} ( \nabla_{Y} \mu,I_{1} ), 
r^{\frac{2}{c}} ( \nabla_{Y} \mu ,I_{2} ), 
r^{\frac{2}{c}} ( \nabla_{Y} \mu ,I_{3} )) \\
&=(r^{\frac{2}{c}} \xi(I_{1}Y), 
r^{\frac{2}{c}} \xi(I_{2}Y), 
r^{\frac{2}{c}} \xi(I_{3}Y)),
\end{align*}
and hence, 
\begin{align*}
d \hat{\mu}_{\alpha} (\hat{I}_{\alpha}^{\bar{\theta},c}\tilde{Y})=- r^{\frac{2}{c}} \xi(Y).
\end{align*}
This shows that $d \hat{\mu}_{\alpha} \circ \hat{I}_{\alpha}^{\bar{\theta},c}$ 
is independent of $\alpha$. When $c=-4(n+1)$, the hypercomplex structure 
$(\hat{I}_{1}^{\bar{\theta},c},\hat{I}_{2}^{\bar{\theta},c},\hat{I}_{3}^{\bar{\theta},c})$ is independent of a choice of 
a quaternionic connection. Therefore, it suffices to check by taking a quaternionic connection 
$\nabla$ such that $\nabla e \otimes \cdots \otimes e=0$ as above. 
It is easy to verify the transversality condition for $\hat{\mu}$.
\end{proof}

\begin{remark}
{\rm 
See \cite{J} for the definitions of quaternionic and hypercomplex moment maps. 
If $\mu$ is a quaternionic moment map, 
then the hypercomplex quotient with respect to $\hat{\mu}$ 
is the Swann bundle of the quaternionic quotient by $\mu$ 
\cite[Proposition 5.1]{J}. 
}
\end{remark}

Take a local section $s: U \to S$ defined on an open set $U \subset M$. 
The pullbacks of $\theta^{i}$ by $s$ to $U$ are denoted by $\theta^{i,U}$, and 
we define the one-forms $\theta^{i,U}_{\alpha}$ by
\[
\theta^{i,U} = s^{\ast} \theta^{i}
= \frac{1}{2} \sum_{\alpha=1}^{3} \theta^{i,U}_{\alpha} e_{\alpha}.
\]
Let $\Omega$ be the curvature form of $\theta$. 
We write
\[
\Omega = \sum_{\alpha=1}^{3} \Omega_{\alpha} e_{\alpha},
\]
and denote its pullback by $s$ by $\Omega^{U}$. 
We then define the two-forms $\Omega_{\alpha}^{U}$ by
\[
\Omega^{U} 
= \frac{1}{2} \sum_{\alpha=1}^{3} \Omega_{\alpha}^{U} e_{\alpha}.
\]

\begin{lemma}\label{5_12}
Let $\nabla$ be a quaternionic connection with $\nabla \nu=0$ for some volume element $\nu$, and let 
$X$ be a vector field preserving $Q$ such that $L_{X} \nabla =0$.  
Setting $\bar{\mu}_{\alpha}:=( (\nabla X)^{Q},I_{\alpha} )$ for $\alpha=1,2,3$, we have 
\begin{align} \label{5_1}
    d \bar{\mu}_{\alpha}
&=-\frac{1}{2} \iota_{X}\Omega^{U}_{\alpha}-\bar{\mu}_{\gamma} \theta^{U}_{\beta} 
+\bar{\mu}_{\beta} \theta^{U}_{\gamma},
\end{align}
where $s=(I_{1},I_{2},I_{3}):U \to S$ is a section defined on $U$ and $\theta^{U}=s^{\ast} \theta$.
\end{lemma}

\begin{proof}
The straightforward calculation leads to the conclusion.  
\end{proof}

\begin{lemma}\label{5_13}
Let $\nabla$ be a quaternionic connection with $\nabla \nu=0$ for some volume element $\nu$, and let 
$X$ be a vector field preserving $Q$ such that $L_{X} \nabla =0$. 
We assume that $Ric^{\nabla}$ is $Q$-hermite. 
If $Ric^{\nabla}_{x} \neq 0$ for any $x \in M$, then 
the solution $(g_{1}, g_{2}, g_{3}):U(\subset M) \to \mathbb{R}$ of \eqref{5_1} is unique (if there exists).   
\end{lemma}
\begin{proof}
It is sufficient to check that the solution of  
$d g_{\alpha}
=-g_{\gamma} \theta^{U}_{\beta} 
+g_{\beta} \theta^{U}_{\gamma}$ 
is only trivial one. Applying the exterior derivative $d$ to it and using $\Omega^{U}_{\alpha}=d \theta^{U}_{\alpha} + \theta^{U}_{\beta} \wedge  \theta^{U}_{\gamma}$, we have 
\begin{align} 
g_{\beta}  \Omega^{U}_{\gamma} - g_{\gamma}  \Omega^{U}_{\beta} =0.  
\end{align}
The assumption for $Ric^{\nabla}$ means that $g_{\alpha}=0$. 
\end{proof}

We have known that the lift $\hat{X}$ of $X$ to $\hat{M}$ is tangent to $S$ (see \cite{CH}).

\begin{lemma}\label{5_14}
Let $\nabla$ be a quaternionic connection, and let 
$X$ be a vector field preserving $Q$ such that $L_{X} \nabla =0$. 
The lift $\hat{X}=\tilde{X} + \sum_{\alpha} f_{\alpha} Z_{\alpha}$ of $X$ to the Swann bundle satisfies 
\begin{align}\label{5_2}
    d f_{\alpha}
&=-\iota_{\tilde{X}}\Omega_{\alpha}-2 f_{\gamma} \theta_{\beta} 
+2 f_{\beta} \theta_{\gamma}. 
\end{align}
\end{lemma}

\begin{proof}
It can be obtained by $L_{\hat{X}} \theta=0$. 
\end{proof}

\begin{lemma}\label{m_fun}
Let $\nabla$ be a quaternionic connection, and denote 
$X$ by a vector field preserving $Q$ such that $L_{X} \nabla =0$.  
Let $(e,\nu)$ is the associated pair. 
If $\nabla \nu=0$ and $Ric^{\nabla}$ is Q-hermite, 
then $\mu=f_{Q}^{\nabla} \otimes e$ is a twistor function. 
\end{lemma}

\begin{proof}
From (1.5.3) and \cite[Corollary 1.5]{AM1}, it holds that 
\begin{align*}
 \mathrm{Tr} R^{\nabla}_{X, I_{1}(\, \cdot \, )} \circ I_{1}
 =\mathrm{Tr} R^{\nabla}_{X, I_{2}(\, \cdot \, )} \circ I_{2}
 =\mathrm{Tr} R^{\nabla}_{X, I_{3}(\, \cdot \, )} \circ I_{3}.
\end{align*}
Therefore (\ref{tw_eq}) holds. 
\end{proof}

Combining the lemmas above, we obtain an explicit expression 
for the lifted map of the twistor function 
$\mu = f_{Q}^{\nabla} \otimes e$, 
which will be used in Section \ref{sec_hp_n}.

\begin{proposition}\label{5_15}
Let $\nabla$ be a quaternionic connection, and let 
$X$ be a vector field preserving $Q$ such that $L_{X} \nabla =0$. 
Let $(e,\nu)$ is the associated pair. 
We assume that $Ric^{\nabla}$ is $Q$-hermite. 
If $Ric^{\nabla}_{x} \neq 0$ for any $x \in M$ and $\nabla \nu=0$, 
then the twistor function $\mu=f_{Q}^{\nabla} \otimes e$ satisfies 
\begin{align} 
\theta(\hat{X}) \circ s = \sum_{\alpha}^{3} \bar{\mu}_{\alpha} e_{\alpha}=
\begin{pmatrix}
   0 & -2 \bar{\mu}_{3} & 2 \bar{\mu}_{2} \\
   2 \bar{\mu}_{3} & 0 & -2 \bar{\mu}_{1}  \\
   -2 \bar{\mu}_{2} & 2\bar{\mu}_{1} & 0
\end{pmatrix}
\,\,\, 
(\theta_{\alpha}(\hat{X}) \circ s = \bar{\mu}_{\alpha}), 
\end{align}
where $s=(I_{1},I_{2},I_{2}):U \to S$ is a local section and 
$\bar{\mu}_{\alpha}:=( f_{Q}^{\nabla},I_{\alpha} )$. 
\end{proposition}

\begin{proof}
The pull-backing \eqref{5_2} by $s$ implies that $f_{\alpha} \circ s$ is the solution of \eqref{5_1}. 
Therefore, by Lemma \ref{5_13}, we have 
$f_{\alpha} \circ s=\bar{\mu}_{\alpha}$ for $\alpha=1,2,3$. 
\end{proof}

The relation between the zero sets of 
$\mu$ and $\hat{\mu}$
is described as follows.

\begin{lemma}\label{5_11}
Let $(M,Q)$ be a quaternionic manifold, and let $\hat{M}$ be the Swann bundle of $M$.  
If $\mu$ is a twistor function of $(M,Q)$, we denote the corresponding function on $\hat{M}$ 
by $\hat{\mu}$ as in Lemma \ref{qh_moment}. Then we have 
\begin{align*}
\mu^{-1}(0)=\hat{\pi}(\hat{\mu}^{-1}(0)), \,\,\, M \setminus \mu^{-1}(0)
=\hat{\pi}(\hat{M} \setminus \hat{\mu}^{-1}(0))
\end{align*}
where $\hat{\pi}:\hat{M} \to M$ is the bundle projection.
\end{lemma}

\subsection{A quaternionic complex manifold}

In this subsection, we recall the notion of a quaternionic complex manifold 
and its fundamental properties (\cite{J}), 
which play an important role in this paper.

\begin{definition}[\cite{J}]
A quaternionic manifold $(M,Q)$ with a twistor function 
$\mu \in \Gamma (Q \otimes L^{s_{0}})$ which vanishes nowhere on $M$
is called a 
{\cmssl quaternionic complex manifold}. 
\end{definition}

For a twistor function $\mu$ that vanishes nowhere on $M$, an almost complex structure is defined as 
\begin{align*}
\frac{\bar{\mu}}{ \| \bar{\mu}  \|}
\end{align*}
by expressing $\mu=\bar{\mu} \otimes e$ for $e \in \Gamma(L^{s_{0}})$. 
This is an almost complex structure determined up to sign. In fact, using
$\mu=\bar{\mu} \otimes e=\bar{\mu}^{\prime} \otimes e^{\prime}$, we have
\begin{align*}
\frac{\bar{\mu}}{ \| \bar{\mu}  \|}=\pm \frac{\bar{\mu}^{\prime}}{ \| \bar{\mu}^{\prime}  \|}. 
\end{align*}

\begin{lemma}\label{a1}
Let $(M,Q)$ be a quaternionic complex manifold with a twistor function 
$\mu \in \Gamma(Q \otimes L^{s_{0}})$ and $\nu$ be a volume form on $M$. 
If $\nabla$ is a quaternionic connection such that $\nabla \nu=0$ and 
$e \in \Gamma(L^{s_{0}})$ with $(e,\nu)$ an associated pair, 
then $\bar{\mu} \in \Gamma(Q)$ given by $\mu=\bar{\mu} \otimes e$ satisfies  
\begin{align*}
\nabla \left( \frac{\bar{\mu}}{ \| \bar{\mu} \|} \right)
=\frac{1}{\| \bar{\mu} \|} \left(
(\xi \circ I_{2}) \otimes I_{2} 
+(\xi \circ I_{3}) \otimes I_{3} 
\right), 
\end{align*}
where $(I_{1}:=\frac{\bar{\mu}}{ \| \bar{\mu} \|},I_{2},I_{2})$ is an admissible frame and 
$\xi$ is given by (\ref{tw_eq}). 
\end{lemma}

\begin{proof}
Since $\mu$ is a twistor  function, then it holds that 
\begin{align*}
\nabla \bar{\mu}=\sum_{\alpha=1}^{3} (\xi \circ I_{\alpha}) \otimes I_{\alpha}. 
\end{align*}
Because $\nabla$ is a metric connection with respect to $(\, \cdot \, , \, \cdot \,)$, 
we have 
\begin{align*}
d \| \bar{\mu} \|^{2} =2 \| \bar{\mu} \| d \| \bar{\mu} \|, \,\,\,
d \langle \bar{\mu}, \bar{\mu} \rangle
=2\langle \nabla \bar{\mu}, \bar{\mu} \rangle
=2\langle \sum_{\alpha=1}^{3} (\xi \circ I_{\alpha}) \otimes I_{\alpha}, \| \bar{\mu} \| I_{1} \rangle
=2 (\xi \circ I_{1}) \| \bar{\mu} \|. 
\end{align*}
So it holds that $d \| \bar{\mu} \|=\xi \circ I_{1}$. 
Consequently, we have 
\begin{align*}
\nabla \left( \frac{\bar{\mu}}{ \| \bar{\mu} \|} \right)
&=d \left(\frac{1}{ \| \bar{\mu} \|} \right) \otimes \bar{\mu}
+\frac{1}{ \| \bar{\mu} \|} \nabla \bar{\mu} \\
&=-\frac{1}{ \| \bar{\mu} \|^{2} } (d \| \bar{\mu} \| ) \otimes (\| \bar{\mu} \| I_{1})
+\frac{1}{ \| \bar{\mu} \|} \sum_{\alpha=1}^{3} (\xi \circ I_{\alpha}) \otimes I_{\alpha} \\
&=-\frac{1}{ \| \bar{\mu} \| } (\xi \circ I_{1}) \otimes I_{1}
+\frac{1}{ \| \bar{\mu} \|} \sum_{\alpha=1}^{3} (\xi \circ I_{\alpha}) \otimes I_{\alpha}\\
&=\frac{1}{\| \bar{\mu} \|} \left(
(\xi \circ I_{2}) \otimes I_{2} 
+(\xi \circ I_{3}) \otimes I_{3} 
\right).
\end{align*}
\end{proof}

\begin{lemma}\label{vol}
Let $(M,Q)$ be a quaternionic complex manifold with a twistor function $\mu$. 
For $\mu=\bar{\mu} \otimes e =\bar{\mu}^{\prime} \otimes e^{\prime}$,
we have 
\begin{align*}
\| \bar{\mu} \|^{-2(n+1)} \nu = \| \bar{\mu}^{\prime} \|^{-2(n+1)} \nu^{\prime},
\end{align*}
where $(e,\nu)$, $(e^{\prime},\nu^{\prime})$ are associated pairs.
\end{lemma}

\begin{proof}
Since $(e,\nu)$, $(e^{\prime},\nu^{\prime})$ are associated pairs, we have
\[ e \otimes \cdots \otimes e=\nu^{\ast}\,\,\, \mbox{and}\,\,\, e^{\prime} \otimes \cdots \otimes e^{\prime} =\nu^{\prime\ast}. \]
If $e=a e^{\prime}$ for a function $a$, we have 
$\nu^{\ast} = e \otimes \cdots \otimes e
=a^{2(n+1)} e^{\prime} \otimes \cdots \otimes e^{\prime}=a^{2(n+1)} \nu^{\prime \ast}$. 
On the other hand, since $(\mu=)\bar{\mu} \otimes e 
=\bar{\mu}^{\prime} \otimes e^{\prime}=(\bar{\mu}^{\prime}/a) \otimes e$, 
we have $\bar{\mu}^{\prime}=a \bar{\mu}$. 
Therefore, 
it holds that 
\begin{align*}
\| \bar{\mu} \|^{-2(n+1)} \nu = | a |^{2(n+1)} \| \bar{\mu}^{\prime} \|^{-2(n+1)} 
a^{-2(n+1)} \nu^{\prime}=\| \bar{\mu}^{\prime} \|^{-2(n+1)} \nu^{\prime}. 
\end{align*}
\end{proof}

For a quaternionic complex manifold $(M,Q)$ with a twistor function $\mu$, we can define 
a specific volume form as in Lemma \ref{vol}, which is denoted by $\nu^{\mu}$. 

\begin{proposition}[\cite{J}]\label{integrable}
Let $(M,Q)$ be a quaternionic complex manifold with a twistor function 
$\mu \in \Gamma(Q \otimes L^{s_{0}})$. 
Denote the almost complex structure defined from $\mu$ by $I$, which is unique up to sign. 
Then there exists a unique quaternionic connection $\nabla$ 
such that $\nabla \nu^{\mu}=0$. This quaternionic connection $\nabla$ 
satisfies $\nabla I=0$, in particular the Ricci tensor of $\nabla$ is symmetric and $I$ is integrable. 
\end{proposition}

\begin{proof}
Theorem 2.4 in \cite{AM1} tells us that 
there exists a unique quaternionic connection $\nabla$ such that $\nabla \nu^{\mu}=0$. 
We can write 
$\mu=\bar{\mu} \otimes e=\frac{\bar{\mu}}{ \| \bar{\mu} \|} \otimes(\| \bar{\mu} \| e )$ and 
see that $(\| \bar{\mu} \| e, \nu^{\mu})$ is an associated pair. 
With respect to $\nabla$, we have
\[ 0=d ( \left\| \frac{\bar{\mu}}{\| \bar{\mu} \| } ) \right\|
)=\xi \circ I_{1}, \] 
which implies $\xi=0$. Here $\xi$ is a corresponding one-form of $\bar{\mu}/ \| \bar{\mu} \|$. 
Therefore, $\nabla I=0$ by Lemma \ref{a1}. 
\end{proof}

\begin{definition}
Let $(M,Q)$ be a quaternionic complex manifold with a twistor function 
$\mu \in \Gamma(Q \otimes L^{s_{0}})$. We denote the unique quaternionic 
connection in Proposition \ref{integrable} by $\nabla^{\mu}$ and call it 
{\cmssl $\mu$-connection}. 
\end{definition}

Therefore, the holonomy group of the $\mu$-connection $\nabla^{\mu}$ is contained in 
$\mathrm{SL}(n,\mathbb{H}) \mathrm{U}(1)$.

\begin{lemma}\label{diff_conn}
Let $(M,Q)$ be a quaternionic complex manifold with a twistor function 
$\mu=\bar{\mu} \otimes e \in \Gamma(Q \otimes L^{s_{0}})$, and let  
$(e,\nu)$ be an associated pair. 
For a (unique) quaternionic connection $\nabla$ such that $\nabla \nu=0$, the difference between 
$\nabla^{\mu}$ and $\nabla$ is given by  
\begin{align*}
\nabla^{\mu}-\nabla=S^{\alpha}, \,\, 
\alpha=- \frac{1}{2} d \log \| \bar{\mu} \|. 
\end{align*}  
\end{lemma}

\begin{proof}
The connection $\nabla^{\mu}$ is determined by the equation 
$\nabla^{\mu} \nu^{\mu} =0$, where $\nu^{\mu} =\| \bar{\mu} \|^{-2(n+1)} \nu$. 
Using $\nabla \nu^{\mu}=d(\| \bar{\mu} \|^{-2(n+1)}) \otimes \nu$ 
and $\mathrm{Tr}S^{\alpha}_{Y}=4(n+1) \alpha(Y)$ for $Y \in TM$, 
we calculate 
\begin{align*}
\nabla^{\mu} \nu^{\mu}
&=d(\| \bar{\mu} \|^{-2(n+1)}) \otimes \nu-(\mathrm{Tr}S^{\alpha}) \otimes \| \bar{\mu} \|^{-2(n+1)} \nu \\
&=d(\| \bar{\mu} \|^{-2(n+1)}) \otimes \nu-4(n+1) \| \bar{\mu} \|^{-2(n+1)} \alpha \otimes \nu. 
\end{align*}
Then the equation $\nabla^{\mu} \nu^{\mu} =0$ holds if and only if 
\[ d(\| \bar{\mu} \|^{-2(n+1)})=4(n+1) \| \bar{\mu} \|^{-2(n+1)} \alpha. \]
This shows the conclusion. 
\end{proof}

It is easy to see that 
\begin{align*}
L_{X} I 
=\nabla_{X} I-[f_{Q}^{\nabla},I]. 
\end{align*}
for any quaternionic connection $\nabla$.

\begin{lemma}\label{Lie_der_I} 
Let $(M,Q)$ be a quaternionic complex manifold with a twistor function 
$\mu \in \Gamma(Q \otimes L^{s_{0}})$ and a quaternionic $S^{1}$-action generated by $X \in \Gamma(TM)$. Then the complex structure $I$ defined by $\mu$ 
is $S^{1}$-invariant 
if and only if $I$ and $f_{Q}^{\nabla^{\mu}}$ are linear dependent.  
\end{lemma}

\begin{remark}
{\rm
Let $(M,Q,g)$ be a quaternionic K\"ahler manifold $(M,Q,g)$ with non-zero scalar curvature on which $S^{1}$ acts isometrically and $\mu$ be a quaternionic K\"ahler moment map. 
The moment map $\mu$ is a unique solution of the equation 
\[ \nabla^{g} \mu=\sum_{i=1}^{3} g(I_{i}X, \, \cdot \,) \otimes I_{i}. \] 
This gives 
\[ \mu= a f_{Q}^{\nabla^{g}}, \]
where $\nabla^{g}$ is the Levi-Civita connection and 
$a$ is a nonzero multiple of the scalar curvature. 
It holds that $\| \mu \|$ is constant along 
$X$ and then $\nabla^{g} I=0$. 
We have $L_{X}I=0$ for the complex structure $I$ induced by $\mu$. 
}
\end{remark}

\begin{corollary}\label{ex_m_map} 
Let $(M,Q)$ be a quaternionic manifold with a twistor function 
$\mu \in \Gamma(Q \otimes L^{s_{0}})$, 
and with a quaternionic $S^{1}$-action generated by $X \in \Gamma(TM)$. 
If the complex structure $I$ defined by $\mu$ is $S^{1}$-invariant, then we have 
\[ \mu=k(f_{Q}^{\nabla^{\mu}}) \otimes e \]
for a function $k$, where $e$ is a section of $L^{s_{0}}$. 
In particular, if $\mu$ is nowhere vanishing and  $f_{Q} \neq 0$ on a connected component of 
the fixed-point set $F$, then 
the induced complex structure on $F$ from $f_{Q}$ coincides with one twistor function $\mu$ up to sign.  
\end{corollary}

\begin{lemma}\label{inv_111}
Let $(M,Q)$ be a quaternionic manifold 
with a quaternionic $S^{1}$-action 
generated by $X \in \Gamma(TM)$. 
Let $I$ be a complex structure compatible with $Q$ such that $L_{X} I=0$, and let
$\nabla$ be a quaternionic connection such that $\nabla I=0$. Then we have $L_{X}\nabla=0$. 
\end{lemma}

\begin{proof}
Consider $\{ \varphi_{t} \}$ be the flow generated by $X$ and connection 
$\nabla^{t}$ defined by 
\[ \nabla^{t}_{Y}Z:=\varphi_{t \ast} (\nabla_{\varphi_{t \ast}^{-1}Y} \varphi_{t \ast}^{-1}Z),  \]
which is a quaternionic connection. Therefore, we can write
\[ \nabla^{t}-\nabla=S^{\xi_{t}} \]
and obtain 
\begin{align}\label{aa1}
L_{X} \nabla =\left. \frac{d}{dt} \nabla^{t} \right|_{t=0}=\left. \frac{d}{dt} S^{\xi_{t}} \right|_{t=0}=S^{\xi_{X}},
\end{align}
where $\xi_{X}=\frac{d}{dt} \xi_{t} |_{t=0}$. 
Since $L_{X}I=0$ and $\nabla I=0$, we have 
\begin{align}\label{aa2}
[L_{X}\nabla,I]=0. 
\end{align}
By \eqref{aa1} and \eqref{aa2}, we have $[S^{\xi_{X}},I]=0$. On the other hand, we know 
$[S^{\xi_{X}},I]=2(\xi_{X} \circ I_{2}) \otimes I_{3}-2(\xi_{X} \circ I_{3}) \otimes I_{2}$ 
with respect to an admissible frame $(I_{1}=I,I_{2},I_{3})$. This shows that $\xi_{X}=0$.  
\end{proof}

By Lemma \ref{inv_111}, we have the following corollary. 

\begin{corollary}\label{inv_mu_conn} 
Let $(M,Q)$ be a quaternionic complex manifold with a twistor function 
$\mu \in \Gamma(Q \otimes L^{s_{0}})$ and a quaternionic $S^{1}$-action generated by $X \in \Gamma(TM)$. 
If the complex structure $I$ defined by $\mu$ is $S^{1}$-invariant, then 
the $\mu$-connection $\nabla^{\mu}$ is 
$S^{1}$-invariant. 
\end{corollary}

\begin{theorem}\label{thm_1a}
Let $(M,Q)$ be a quaternionic manifold with a twistor function 
$\mu=\bar{\mu} \otimes e \in \Gamma(Q \otimes L^{s_{0}})$. 
Suppose that $S^{1}$ acts on $M$ preserving $Q$ and a complex structure $I$ given by $\mu$ is invariant by  the $S^{1}$-action.
Let $F$ be a connected component of the fixed-point set of the $S^{1}$-action with $f_{Q} \neq 0$ 
and $F \cap \mu^{-1}(0) = \emptyset$. 
The induced connections on $F$ from $\nabla^{\mu}$ and  
an $S^{1}$-invariant quaternionic connection $\nabla$ such that $\nabla \nu=0$ coincide on $F$, 
where $(e,\nu)$ is an associated pair.   
\end{theorem}

\begin{proof}
By Lemma \ref{diff_conn}, the difference between 
$\nabla^{\mu}$ and $\nabla$ is given by  
\begin{align*}
\nabla^{\mu}-\nabla=S^{\alpha}, \,\, 
\alpha=- \frac{1}{2} d \log \| \bar{\mu} \|, 
\end{align*}
where $\mu=\bar{\mu} \otimes e$ and $(e,\nu)$ is the associated pair. 
Since $L_{X}\nabla=0$, we have $\nabla f_{Q}^{\nabla}=0$ on $F$ by Lemma \ref{der_nab_q}.
Also $\nabla^{\mu}$ is $S^{1}$-invariant due to Corollary \ref{inv_mu_conn}, which leads that 
$\nabla^{\mu} f_{Q}^{\nabla^{\mu}}=0$ on $F$. 
Since $f_{Q}^{\nabla}=f_{Q}^{\nabla^{\mu}}$ on $F$, we have
\begin{align*}
0 &= \nabla^{\mu} f_{Q}^{\nabla^{\mu}} = \nabla^{\mu} f_{Q}^{\nabla}
    =(\nabla+S^{\alpha})  f_{Q}^{\nabla}=[ S^{\alpha},  f_{Q}^{\nabla} ] 
    =2 \| f_{Q}^{\nabla} \| ((\alpha \circ I_{2}) I_{3} - (\alpha \circ I_{3}) I_{2}),
\end{align*}
where $I_{1}=\frac{f_{Q}^{\nabla}}{\| f_{Q}^{\nabla} \|}$ and $(I_{1},I_{2},I_{3})$ is an admissible frame. Then it holds that $\alpha=0$ on $F$. 
Therefore, the induced connections on $F$ from $\nabla^{\mu}$ and $\nabla$ coincide. 
\end{proof}

The following corollary will be used later to investigate the nonexistence of (locally) hypercomplex 
structures from fixed-point sets (Corollary \ref{QK_Fix}). 

\begin{corollary}\label{3.22}
Let $(M,Q,g)$ be a quaternionic K\"ahler manifold. 
Suppose that $S^{1}$ acts on $M$ isometrically, 
and let $F$ be a connected component of the fixed-point set 
such that $f_{Q} \neq 0$ and $F \cap \mu^{-1}(0)=\emptyset$, 
where $\mu$ denotes the quaternionic K\"ahler moment map 
associated with the $S^{1}$-action. 
Then the connection on $F$ induced by $\nabla^{\mu}$ 
coincides with the Levi-Civita connection of $F$.
\end{corollary}

\begin{proof}
The Levi-Civita connection $\nabla^{g}$ is $S^{1}$-invariant, 
and the volume form induced by $g$ is parallel with respect to $\nabla^{g}$.
\end{proof}

\subsection{A compatible complex structure of closed type}

For any admissible frame $(I_{1},I_{2},I_{3})$ of $Q$, there are connection forms $\omega^{\nabla}_{1}$,  
$\omega^{\nabla}_{2}$, $\omega^{\nabla}_{3}$ with respect to a quaternionic connection $\nabla$ with 
\begin{align}\label{conn_f}
\nabla I_{\alpha}=\omega^{\nabla}_{\gamma} \otimes I_{\beta} - \omega^{\nabla}_{\beta} \otimes I_{\gamma}, 
\end{align}
where $(\alpha,\beta,\gamma)$ is any cyclic permutation of $(1,2,3)$. It is known that 
$I_{1}$ is integrable if and only if $\omega^{\nabla}_{2} \circ I_{2}=\omega^{\nabla}_{3} \circ I_{3}$ for any 
admissible frame $(I_{1},I_{2},I_{3})$. The integrability of $I_{1}$ is independent of the choice of 
a quaternionic connection $\nabla$. Hence $\omega^{\nabla}_{2} \circ I_{2}=\omega^{\nabla}_{3} \circ I_{3}$ holds for any quaternionic connection if $I_{1}$ is integrable. 
 
\begin{lemma}\label{lie}
For admissible frames $(I_{1},I_{2},I_{3})$ and $(\bar{I}_{1},\bar{I}_{2},\bar{I}_{3})$, 
if $I_{1}=\bar{I}_{1}$, then we have 
\[ \omega^{\nabla}_{2} \circ I_{2}+ \omega^{\nabla}_{3} \circ I_{3}
=\bar{\omega}^{\nabla}_{2} \circ \bar{I}_{2}+ \bar{\omega}^{\nabla}_{3} \circ \bar{I}_{3}, \]
where $\omega^{\nabla}_{\alpha}$ (resp. $\bar{\omega}^{\nabla}_{\alpha}$) is 
the connection form with respect to $(I_{1},I_{2},I_{3})$ (resp. $(\bar{I}_{1},\bar{I}_{2},\bar{I}_{3})$)
given by \eqref{conn_f}.  
\end{lemma}

\begin{proof}
For simplicity, we omit $\nabla$ for the subscript of $\omega_{\alpha}^{\nabla}$ ($\alpha=2, 3$).
Since 
\begin{align*}
&\nabla I=\omega_{2} \otimes I_{3}-\omega_{3} \otimes I_{2}
=\bar{\omega}_{2} \otimes \bar{I}_{3}-\bar{\omega}_{3} \otimes \bar{I}_{2}, \\
&I_{2} =\cos \theta \bar{I}_{2}-\sin \theta  \bar{I}_{3}, \,\,\,
I_{3} =\sin \theta \bar{I}_{2}+\cos \theta  \bar{I}_{3},
\end{align*}
we have 
\begin{align*}
\omega_{2} =\cos \theta \bar{\omega}_{2}-\sin \theta  \bar{\omega}_{3}, \,\,\,
\omega_{3} =\sin \theta \bar{\omega}_{2}+\cos \theta  \bar{\omega}_{3}.
\end{align*}
Then it holds that 
\begin{align*}
\omega_{2} \circ I_{2}
&=\cos^{2} \theta (\bar{\omega}^{\nabla}_{2} \circ \bar{I}_{2}) 
-\sin \theta \cos \theta (\bar{\omega}^{\nabla}_{2} \circ \bar{I}_{3}) 
-\sin \theta \cos \theta (\bar{\omega}^{\nabla}_{3} \circ \bar{I}_{2})
+ \sin^{2} \theta (\bar{\omega}^{\nabla}_{3} \circ \bar{I}_{3}),  \\
\omega_{3} \circ I_{3}
&=\sin^{2} \theta (\bar{\omega}^{\nabla}_{2} \circ \bar{I}_{2}) 
+\sin \theta \cos \theta (\bar{\omega}^{\nabla}_{2} \circ \bar{I}_{3}) 
+\sin \theta \cos \theta (\bar{\omega}^{\nabla}_{3} \circ \bar{I}_{2})
+ \cos^{2} \theta (\bar{\omega}^{\nabla}_{3} \circ \bar{I}_{3}). 
\end{align*}
This completes the proof. 
\end{proof}

By virtue of Lemma \ref{lie}, for a compatible almost complex structure $I$ and a quaternionic connection 
$\nabla$, we can define a one-form
\[ \eta^{\nabla}_{I}:=\omega^{\nabla}_{2} \circ I_{2}+ \omega^{\nabla}_{3} \circ I_{3} \]
on $M$. 
Since $\omega^{\nabla^{\prime}}_{\alpha}=\omega^{\nabla}_{\alpha} - 2 \xi \circ I_{\alpha}$ 
($\alpha=1,2,3$), we have $\eta^{\nabla^{\prime}}_{I}=\eta^{\nabla}_{I}+4 \xi$. 
As Remark 5.2 in \cite{AM1}, 
if both Ricci tensors of quaternionic connections $\nabla$ and $\nabla^{\prime}$ with  
$\nabla^{\prime}=\nabla+S^{\xi}$ are symmetric, then $\xi$ is exact. 
Hence, we conclude as follows. 

\begin{lemma}\label{3_29}
Let $I$ be a compatible almost complex structure, and let $\nabla$, $\nabla^{\prime}$
be quaternionic connections with symmetric Ricci tensors. Then we have 
\[ d \eta^{\nabla^{\prime}}_{I}=d \eta^{\nabla}_{I}.  \]
In particular, the property that  $\eta^{\nabla}_{I}$ is closed is independent of a choice of 
a Ricci symmetric quaternionic connection.
\end{lemma}   

\begin{proof}
Any two Ricci symmetric quaternionic connections are related by an exact one-form in 
\eqref{q_conn}. 
\end{proof}

Lemma \ref{3_29} leads the following definition.

\begin{definition}\label{def_close_type}
Let $(M,Q)$ be a quaternionic manifold, and let $I$ be a compatible almost complex structure with $Q$. 
We say that $I$ is {\cmssl closed type}
if $\eta^{\nabla}_{I}$ is closed for one 
(and hence any) Ricci symmetric quaternionic connection $\nabla$. 
\end{definition}

For a quaternionic complex manifold $(M,Q)$ given by a twistor function $\mu$, 
the Ricci tensor of $\mu$-connection $\nabla^{\mu}$ is symmetric and 
$\eta^{\nabla^{\mu}}_{I}=0$, where $I$ is given by $\mu$. 
Therefore, if $I$ is not of closed type, then 
there exists no twistor function $\mu$ so that $\mu$ induces $I$.

If $I$ is integrable, then 
$\omega^{\nabla}_{2} \circ I_{2}=\omega^{\nabla}_{3} \circ I_{3}(=(1/2) \eta^{\nabla}_{I})$
is also a globally defined one-form. In fact, for $(I_{1}=I, I_{2}, I_{3})$ (resp. ($\bar{I}_{1}=I, \bar{I}_{2}, \bar{I}_{3})$) 
is an admissible frame defined on $U$ (resp. $V$), we see that 
\[ \omega^{\nabla}_{2} \circ I_{2}=\omega^{\nabla}_{3} \circ I_{3}=\frac{1}{2}\eta^{\nabla}_{I} =\bar{\omega}^{\nabla}_{2} \circ \bar{I}_{2}=\bar{\omega}^{\nabla}_{3} \circ \bar{I}_{3} \]
on $U \cap V$.

\begin{lemma}\label{ex_tw}
Let $(M,Q)$ be a quaternionic manifold with a volume form $\nu$ 
and $I$ be a complex structure compatible with $Q$. 
Let $\nabla$ be a quaternionic connection such that $\nabla \nu=0$ and 
$(e,\nu)$ be the associated pair. 
If a section $\mu=f I \otimes e \in \Gamma(Q \otimes L^{s_{0}})$ is 
a nowhere vanishing twistor function if and only if $\eta_{I}^{\nabla}=2 d (\log |f|) $ 
for a nowhere vanishing function $f$.   
\end{lemma}

\begin{proof}
Suppose that $\mu=f I \otimes e $ is a twistor function. Then we have 
\begin{align*}
\nabla (f I \otimes e)=df \otimes I+ f(\nabla I)
=(\xi \circ I) \otimes I+(\xi \circ I_{2}) \otimes I_{2}+(\xi \circ I_{3}) \otimes I_{3},
\end{align*}
where $(I_{1}=I,I_{2},I_{3})$ is an admissible frame. This implies that 
$f \omega_{2} \circ I_{2}=f \omega_{3} \circ I_{3}=df$. Then we have $\eta_{I}^{\nabla}=2 d (\log |f|) $.
Conversely, setting $\xi=-df \circ I$, we have
\ \begin{align*}
\nabla (f I \otimes e)&=df \otimes I+ f(\nabla I) \\
&=df \otimes I+f(\omega^{\nabla}_{3} \otimes I_{2} - \omega^{\nabla}_{2} \otimes I_{3}) \\
&=df \otimes I
+f \left( \left(-\frac{df}{f} \circ I_{3} \right) \otimes I_{2}
             +\left(\frac{df}{f} \circ I_{2} \right) \otimes I_{3} \right) \\
&=(\xi \circ I) \otimes I +(\xi \circ I_{2}) \otimes I_{2}+(\xi \circ I_{3}) \otimes I_{3},
\end{align*}
which means that $\mu$ is a twistor function. 
\end{proof}

\begin{lemma}\label{3_33}
Let $(M,Q)$ be a quaternionic manifold with a compatible complex structure $I$. 
If $\nabla$ and $\nabla^{\prime}$ are quaternionic connections such that $\nabla I=\nabla^{\prime} I=0$, then $\nabla=\nabla^{\prime}$.  
\end{lemma}

\begin{proof}
The difference tensor $S^{\xi}$ for $\nabla$ and $\nabla^{\prime}$ 
in Lemma \ref{q_conn_eq} satisfies that $[S^{\xi} ,I]=0$. This means that $\xi=0$. 
\end{proof}

The following proposition means that a compatible complex structure $I$ of closed type 
has a collection of quaternionic complex manifolds and twistor functions. 

\begin{proposition}\label{th_2a}
Let $(M,Q)$ be a quaternionic manifold, and let $I$ be a compatible complex structure 
with $Q$. Suppose that an $S^{1}$-action on $M$ preserves both $Q$ and $I$.
If $I$ is of closed type, then there exists a unique quaternionic connection $\nabla^{I}$ 
such that $\nabla^{I} I=0$. Moreover, the complex structure $I$ is 
induced from a locally defined twistor function $\mu$, and $\nabla^{I}$ locally coincides with 
the $\mu$-connection.  
\end{proposition}

\begin{proof}
Take an associated pair $(e,\nu)$ and $\nabla$ be a quaternionic connection on $(M,Q)$ 
such that $\nabla \nu=0$.
For any point $x \in M$, there exists an open set $U \subset M$ containing $x$ 
on which $\eta_{I}^{\nabla}$ is exact due to closedness of $I$. 
Choosing a function $f^{U}$ on $U$ such that $\eta_{I}^{\nabla}=2 d \log |f^{U}|$ and $f^{U}$ is nowhere vanishes, 
we obtain a twistor function 
$\mu^{U}=f^{U}I \otimes e$ on $U$ by Lemma \ref{ex_tw}. 
We obtain 
$\mu^{U}$-connection $\nabla^{\mu^{U}}$ on $U$. 
Considering this procedure for each point of $M$, we obtain a collection of $\mathrm{SL}(n,\mathbb{H}) \mathrm{U}(1)$-connections. 
Since $\nabla^{\mu^{U}}$ are related to $\nabla|_{U}$ on each $U$ by \eqref{q_conn},
the connections $\nabla^{\mu^{U}}$ and $\nabla^{\mu^{V}}$ are related to 
\[ 
\nabla^{\mu^{V}}=\nabla^{\mu^{U}}+S^{\alpha}, \,\,\,
\alpha=-\frac{1}{2} d \log \left| \frac{ f^{U} }{ f^{V} } \right|
\]
on $U \cap V$ by Lemma \ref{diff_conn}. Because $\nabla^{\mu^{V}} I =\nabla^{\mu^{U}} I=0$ on $U \cap V$, 
we have $\alpha=-\frac{1}{2} d \log \left| \frac{ f^{U} }{ f^{V} } \right|=0$. This defines 
the quaternionic connection $\nabla^{I}$ with $\nabla^{I} I=0$. 
The uniqueness of such a connection follows from Lemma \ref{3_33}. 
\end{proof}

By \cite[Corollary 7.2]{J} and \cite[Theorem 2.1]{P2}, a quaternionic complex manifold $M$ 
of $\dim M=4$ is exactly a K\"ahler surface with zero scalar curvature.
Therefore, when $\dim M=4$, 
the closed type condition for a compatible complex structure $I$ with $Q$ means that 
$M$ admits an open covering $\{ U_{i} \}$ of $M$ such that each $(U_{i},Q|_{U_{i}})$ 
is conformal to a K\"ahler manifold of zero scalar curvature.

\section{An obstruction to the existence of hypercomplex structures}
\setcounter{equation}{0}

\subsection{Curvatures of the $\mu$-connection}\label{sec_mu_conn}
As (1.5.2) in \cite{AM1}, we have 
\begin{align*}
 [ R^{\nabla}_{X,Y}, I_{\alpha} ] = \Omega^{\nabla}_{\gamma}(X,Y)I_{\beta}-\Omega^{\nabla}_{\beta}(X,Y)I_{\gamma}
\end{align*}
for $X$, $Y \in TM$, where $(\alpha,\beta,\gamma)$ is any cyclic permutation of $(1,2,3)$. 
From \cite[Corollary 1.4 and Remark 1.12]{AM1}, 
the two-form $\Omega_{\alpha}$ is 
given by 
\begin{align*}
 \Omega^{\nabla}_{\alpha}(X,Y)=2(B^{\nabla}(X,I_{\alpha}Y)-B^{\nabla}(Y,I_{\alpha}X))
\end{align*}
for $X$, $Y \in TM$. On the other hand, we have 
\begin{align*}
 \Omega^{\nabla}_{\alpha}=d \omega^{\nabla}_{\alpha}+\omega^{\nabla}_{\beta} \wedge \omega^{\nabla}_{\gamma}. 
\end{align*}

We consider the case of a quaternionic complex manifold defined by a nowhere vanishing twistor function $\mu$. 
Take an admissible frame $(I_{1}=I,I_{2},I_{3})$, where $I$ is a complex structure defined from $\mu$. 
Since $\nabla^{\mu} I=0$, we have $\omega^{\nabla^{\mu}}_{2}=\omega^{\nabla^{\mu}}_{3}=0$, and hence 
$\Omega^{\nabla^{\mu}}_{2}=\Omega^{\nabla^{\mu}}_{3}=0$.   

\begin{lemma}
The Ricci tensor $Ric^{\nabla^{\mu}}$ satisfies that 
\[ Ric^{\nabla^{\mu}}(X,Y)=Ric^{\nabla^{\mu}}(I_{1}X,I_{1}Y)
=-Ric^{\nabla^{\mu}}(I_{2}X,I_{2}Y)=-Ric^{\nabla^{\mu}}(I_{3}X,I_{3}Y)
\]
for $X$, $Y \in TM$.
\end{lemma}
\begin{proof}
Since $\Omega^{\nabla^{\mu}}_{2}=0$ and $Ric^{\nabla^{\mu}}$ is symmetric, it holds that 
$B^{\nabla^{\mu}}(X,Y)=-B^{\nabla^{\mu}}(I_{2}X,I_{2}Y)$
for any $X$, $Y \in TM$. 
Then we have
\begin{align*}
\frac{1}{4n}Ric^{\nabla^{\mu}}(X,Y)
&-\frac{1}{2n(n+2)} \Pi_{h} Ric^{\nabla^{\mu}}(X,Y) \\
&=-\frac{1}{4n}Ric^{\nabla^{\mu}}(I_{2}X,I_{2}Y)
+\frac{1}{2n(n+2)} \Pi_{h} Ric^{\nabla^{\mu}}(I_{2}X,I_{2}Y) \\
&=-\frac{1}{4n}Ric^{\nabla^{\mu}}(I_{2}X,I_{2}Y)
+\frac{1}{2n(n+2)} \Pi_{h} Ric^{\nabla^{\mu}} (X,Y).
\end{align*}
This means that 
\begin{align}\label{a111}
Ric^{\nabla^{\mu}}(X,Y)+Ric^{\nabla^{\mu}}(I_{2}X,I_{2}Y)
=\frac{2}{(n+2)} \Pi_{h} Ric^{\nabla^{\mu}}(X,Y).
\end{align}
Similarly, we have  
\begin{align}\label{a112}
Ric^{\nabla^{\mu}}(X,Y)+Ric^{\nabla^{\mu}}(I_{3}X,I_{3}Y)
=\frac{2}{(n+2)} \Pi_{h} Ric^{\nabla^{\mu}}(X,Y).
\end{align}
Combining these two equations \eqref{a111} and \eqref{a112}, it holds that 
\begin{align*}
Ric^{\nabla^{\mu}}(I_{2}X,I_{2}Y)=Ric^{\nabla^{\mu}}(I_{3}X,I_{3}Y), 
\end{align*}
which also implies that 
\begin{align*}
Ric^{\nabla^{\mu}}(X,Y)=Ric^{\nabla^{\mu}}(I_{1}X,I_{1}Y). 
\end{align*}
Again by \eqref{a111}, we see that 
\begin{align*}
Ric^{\nabla^{\mu}}(X,Y)+Ric^{\nabla^{\mu}}(I_{2}X,I_{2}Y)
=\frac{2}{(n+2)} \cdot \frac{1}{4} \left( 2Ric^{\nabla^{\mu}}(X,Y)+2Ric^{\nabla^{\mu}}(I_{2}X,I_{2}Y) \right).
\end{align*}
Therefore, we have $Ric^{\nabla^{\mu}}(X,Y)=-Ric^{\nabla^{\mu}}(I_{2}X,I_{2}Y)$. 
\end{proof}

\begin{corollary}\label{cor_42}
$\Pi_{h} Ric^{\nabla^{\mu}}=0$ and $B^{\nabla^{\mu}}=\frac{1}{4n} Ric^{\nabla^{\mu}}$.
\end{corollary}

We set $Ric^{\nabla^{\mu}}_{I_{1}}(X,Y):=Ric^{\nabla^{\mu}}(X,I_{1}Y)$ for $X$, $Y \in TM$.

\begin{lemma}
The non-vanishing coefficient $\Omega^{\nabla^{\mu}}_{1}$ is given by 
\begin{align*}
\Omega^{\nabla^{\mu}}_{1}=\frac{1}{n} Ric^{\nabla^{\mu}}_{I_{1}}. 
\end{align*}
\end{lemma}

\begin{proof}
By corollary \ref{cor_42}, we have 
\begin{align*}
\Omega^{\nabla^{\mu}}_{1}(X,Y) 
&=2(B^{\nabla^{\mu}}(X,I_{1}Y)-B^{\nabla^{\mu}}(Y,I_{1}X)) \\
&=2 \left( \frac{1}{4n} Ric^{\nabla^{\mu}} (X,I_{1}Y)- \frac{1}{4n} Ric^{\nabla^{\mu}} (I_{1}Y,X) \right) \\
&=\frac{1}{n} Ric^{\nabla^{\mu}} (X,I_{1}Y)=\frac{1}{n} Ric^{\nabla^{\mu}}_{I_{1}}(X,Y)
\end{align*}
for all $X$, $Y \in TM$. 
\end{proof}

Moreover, we have 
\begin{align*}
\mathrm{Tr} (I_{1} R^{\nabla^{\mu}}_{X,Y})=-2Ric^{\nabla^{\mu}}_{I_{1}}(X,Y)
\end{align*}
since 
\[ \Omega^{\nabla^{\mu}}_{1}(X,Y) = -\frac{1}{2n} \mathrm{Tr} (I_{1} R^{\nabla^{\mu}}_{X,Y}). \]

The following proposition shows that 
the Ricci tensor $Ric^{\nabla^{\mu}}$ of the $\mu$-connection 
provides an obstruction to the existence of a local hypercomplex structure 
of the form $(I_{1}=I,I_{2},I_{3})$.

\begin{proposition}\label{local_obst}
Let $(M,Q)$ be a quaternionic complex manifold with a twistor function $\mu$, 
and let $I$ be the complex structure induced by $\mu$.  
If $(I_{1}=I,I_{2},I_{3})$ is a hypercomplex structure on an open set $U \subset M$, 
then $Ric^{\nabla^{\mu}}=0$ on $U$. 
In particular, if $Ric^{\nabla^{\mu}}_{x} \neq 0$ at a point $x \in M$, 
then any neighborhood $V$ of $x$ does not admit a hypercomplex structure
$(I_{1}=I,I_{2},I_{3})$ on $V$.
\end{proposition}

\begin{proof}
Let $\nabla^{0}$ be the Obata connection on $U$ associated with 
the hypercomplex structure $(I_{1}=I,I_{2},I_{3})$. 
Since $\nabla^{0}I=\nabla^{\mu}I=0$, 
Lemma \ref{3_33} implies that $\nabla^{0}=\nabla^{\mu}$. 
Hence $Ric^{\nabla^{0}}=Ric^{\nabla^{\mu}}$. 
Because $Ric^{\nabla^{0}}$ is skew-symmetric whereas 
$Ric^{\nabla^{\mu}}$ is symmetric, 
it follows that $Ric^{\nabla^{\mu}}=0$ on $U$.
\end{proof}

To investigate the global existence of a hypercomplex structure 
of the form $(I_{1}=I,I_{2},I_{3})$, 
we make use of the first Chern class. 
For this purpose, we prove the following theorem.

\begin{theorem}\label{chern_mu_conn}
The first Chern class $c_{1}(M)$ of a quaternionic complex manifold $(M,I_{1})$ defined by twistor function 
$\mu$ is represented by the Ricci form of the $\mu$-connection. More precisely,
\begin{align*}
c_{1}(M)=\left[ -\frac{1}{2 \pi} Ric^{\nabla^{\mu}}_{I_{1}} \right]. 
\end{align*}
\end{theorem}

The {\cmssl characteristic $4$-form $\Theta^{\nabla}$} associated 
with the quaternionic connection $\nabla$ is defined by 
\begin{align*}
\Theta^{\nabla}=\sum_{\alpha=1}^{3} \Omega^{\nabla}_{\alpha} \wedge \Omega^{\nabla}_{\alpha}
\end{align*}
as in \cite{AM1}. 
Note that $d \Theta^{\nabla}=0$ (see \cite[Proposition 5.2]{AM1}). 
In particular, if the principal $\mathrm{SO}(3)$-bundle $S(M)$ associated with $Q$ which is the bundle of 
admissible frames is trivial, 
then $\Theta^{\nabla}$ is exact 
for any quaternionic connection $\nabla$ 
due to \cite[Propositions 5.2 and 5.3]{AM1}.

\begin{proposition}\label{a113}
For a quaternionic complex manifold $M$, the characteristic $4$-form $\Theta^{\nabla^{\mu}}$ satisfies  
$4  \pi^{2} c_{1}(M)^{2}=n^{2}[\Theta^{\nabla^{\mu}}]$.
\end{proposition}

\subsection{Proof of the main theorem}

The following lemma is easy to prove. 

\begin{lemma}\label{hyp_chern}
Let $(M,(I_{1},I_{2},I_{3}))$ be a hypercomplex manifold. The 
the first Chern class of $(M,I_{\alpha})$ vanishes for $\alpha=1,2,3$. 
\end{lemma}
\begin{proof}
Let $\nabla^{0}$ be the Obata connection. We have $[R^{\nabla^{0}},I_{\alpha}]=0$, and hence 
$\Omega_{\alpha}=0$ for $\alpha=1,2,3$. Therefore we see that $\mathrm{Tr} I_{\alpha} R^{\nabla^{0}}_{X,Y}=0$.  
\end{proof}

If $I$ is a compatible complex structure of closed type on a quaternionic manifold, then the Ricci tensor 
$Ric^{\nabla^{I}}$ of the quaternionic connection 
$\nabla^{I}$ given in Proposition \ref{th_2a} satisfies the same properties as those of $\mu$-connection. 
Precisely, we have 

\begin{lemma}\label{Ric_Iconn}
Let $I$ be a compatible complex structure of closed type. 
The Ricci tensor $Ric^{\nabla^{I}}$ is symmetric and
\begin{align*}
Ric^{\nabla^{I}}(X,Y)=Ric^{\nabla^{I}}(I_{1}X,I_{1}Y)
=-Ric^{\nabla^{I}}(I_{2}X,I_{2}Y)=-Ric^{\nabla^{I}}(I_{3}X,I_{3}Y)
\end{align*}
for any $X$, $Y \in TM$, where $(I_{1}=I,I_{2},I_{3})$ is an admissible frame. Hence we have
$\Pi_{h} Ric^{\nabla^{I}}=0$ and $B^{\nabla^{I}}=\frac{1}{4n} Ric^{\nabla^{I}}$. 
In particular, the similar statement as Proposition \ref{local_obst} holds for $\nabla^{I}$. 
Moreover, the first Chern class $c_{1}(M)$ of $(M,I)$ is given by
\begin{align*}
c_{1}(M)=\left[ -\frac{1}{2 \pi} Ric^{\nabla^{I}}_{I} \right]. 
\end{align*}
\end{lemma}

Now we prove the main theorem of this paper. 

\begin{theorem}\label{q_flat_fix}
Suppose that $(M,Q)$ is a quaternionic manifold of $\dim M=4n$ 
endowed with an $S^{1}$-quaternionic action, and 
that $F$ is a connected component of the fixed-point set of the $S^{1}$-action with $f_{Q} \neq 0$. 
Let $I$ be an $S^{1}$-invariant compatible complex structure of closed type. 
If the quaternionic Weyl curvature $W$ of $(M,Q)$ vanishes, then we obtain  
\[ 2n c_{1}(F) = (m+1) \iota^{\ast} c_{1}(M), \]
where $2m=\dim F$ and $\iota:F \to M$ denotes the inclusion.  
In particular, this yields the following obstructions
\begin{enumerate}
  \item if  $c_{1}(F) \neq 0$, then there exists no hypercomplex structure $(I_{1},I_{2},I_{3})$ 
  on any open set $U$ of $M$ such that $I_{1}=I$ and $F \subset U$. 
  \item if  $c_{1}(F)^{2} \neq 0$, then the ${\rm SO(3)}$-bundle $S(U)$ is not trivial 
  for any open set $U$ of $M$ such that $F \subset U$. In particular, if $c_{1}(F)^{2} \neq 0$, then 
  there exists no hypercomplex structure $H$ on a connected component of $M$ containing $F$ 
  such that $Q=\langle H \rangle$. 
\end{enumerate}
\end{theorem}

\begin{proof}
From the assumption, we have the quaternionic connection $\nabla^{I}$ such that $\nabla^{I}I=0$. 
Since $L_{X} I = 0$ and $\nabla^{I} I = 0$, 
it follows that $f_{Q}$ and $I$ are linearly dependent, hence $
I = \pm \frac{f_{Q}}{\| f_{Q} \|}$ on $F$.
Therefore, by (3) of Lemma \ref{cpx_submfd}, 
$F$ is transversally complex. 
By Lemma \ref{inv_111}, 
the connection $\nabla^{I}$ is $S^{1}$-invariant. 
Because of Lemma \ref{fix_point}, $F$ is totally geodesic with respect to $\nabla^{I}$. 
Consequently, we can define a connection $\nabla^{F}$ on $F$ by 
\[ \iota_{\ast}(\nabla^{F}_{X}Y)= \nabla^{I}_{X} \iota_{\ast}Y \]
for $X$, $Y \in \Gamma(TF)$. Then 
it holds that $\iota_{\ast} R^{\nabla^{F}}_{X,Y}Z=R^{\nabla^{I}}_{\iota_{\ast}X, \iota_{\ast}Y} \iota_{\ast}Z$.  
Since $W=0$, we obtain 
\[ R^{\nabla^{I}}_{\iota_{\ast}X, \iota_{\ast}Y} \iota_{\ast}Z=R^{B^{\nabla^{I}}}_{\iota_{\ast}X, \iota_{\ast}Y} \iota_{\ast}Z \in \iota_{\ast}(TF) \oplus I_{2} (\iota_{\ast}(TF)) 
(=\iota_{\ast}(TF) \oplus I_{3} (\iota_{\ast}(TF))) \]
for $X$, $Y$, $Z \in TF$. 
Denote the projection from 
$\iota_{\ast}(TF) \oplus I_{2} (\iota_{\ast}(TF)) $ on the first factor $\iota_{\ast}(T (F \cap U))$ by $p$, 
where $(I_{1}=I,I_{2},I_{3})$ is an admissible frame. 
We have 
\begin{align}\label{4_10_ric}
\iota_{\ast} R^{\nabla^{F}}_{X,Y}Z
=&p(R^{\nabla^{I}}_{\iota_{\ast}X, \iota_{\ast}Y} \iota_{\ast}Z)
=p(R^{B^{\nabla^{I}}}_{\iota_{\ast}X, \iota_{\ast}Y} \iota_{\ast}Z) \\
=&B^{\nabla^{I}}( \iota_{\ast}Y, \iota_{\ast}Z) \iota_{\ast}X
-B^{\nabla^{I}}( \iota_{\ast}X, \iota_{\ast}Z) \iota_{\ast}Y 
+2 B^{\nabla^{I}}( \iota_{\ast}X,I_{1} \iota_{\ast}Y)I_{1} \iota_{\ast}Z \nonumber \\
&-B^{\nabla^{I}}( \iota_{\ast}Y,I_{1} \iota_{\ast}Z)I_{1}X
+B^{\nabla^{I}}( \iota_{\ast}X,I_{1} \iota_{\ast}Z)I_{1} \iota_{\ast}Y \nonumber
\end{align}
for $X$, $Y$, $Z \in TF$ by the virtue of 
Remark \ref{rem_para}. 
Since $B^{\nabla^{I}}=\frac{1}{4n} Ric^{\nabla^{I}}$, we have
\[ n \mathrm{Tr}R^{\nabla^{F}}_{X,Y} I_{1}=(m+1)(\iota^{\ast}Ric^{\nabla^{I}})(I_{1}X,Y) \]
for $X$, $Y \in TF$ on $F$. This shows that 
\begin{align*}
2n c_{1}(F) =(m+1) \iota^{\ast}(c_{1}(M)). 
\end{align*}
Theorem \ref{chern_mu_conn} together with 
Lemmas \ref{hyp_chern} and \ref{Ric_Iconn} 
implies that $M$ does not admit a hypercomplex structure 
$(I_{1},I_{2},I_{3})$ on any open set $U \subset M$ 
with $I_{1}=I$ and $F \subset U$, 
since $c_{1}(F) \neq 0$. 
The second assertion follows from Proposition \ref{a113} 
and Lemma \ref{Ric_Iconn}, 
as the characteristic $4$-form is not exact.
\end{proof}

\begin{remark}
{\rm 
If $Ric^{\nabla^{F}}_{x} \neq 0$ at a point $x \in F$, then any neighborhood $U \subset M$ of $x$ does 
not admit a hypercomplex structure $(I_{1}=I,I_{2},I_{3})$ on $U$ by 
Proposition \ref{local_obst}, Lemma \ref{Ric_Iconn} 
and \eqref{4_10_ric}. 
}
\end{remark}

As a direct consequence, we have the following corollary.

\begin{corollary}\label{q_flat_fix_qcm}
Let $(M,Q)$ be a quaternionic manifold with an $S^{1}$-quaternionic action and a twistor function $\mu$, and let $F$ be a connected component of the fixed-point set of the $S^{1}$-action with $f_{Q} \neq 0$ 
and $\mu^{-1}(0) \cap F =\emptyset$. 
Set $M^{\prime}=M \setminus \mu^{-1}(0)$. 
Assume that the complex structure $I$ induced by $\mu$ is $S^{1}$-invariant and that 
the quaternionic Weyl curvature of $(M,Q)$ vanishes. 
Then the same conclusion as in Theorem \ref{q_flat_fix} holds for $M^{\prime}$ and $F$.
\end{corollary}

\begin{proof}
The complex structure $I$ induced by the twistor function $\mu$ 
is of closed type.
\end{proof}

We point out that there is a non quaternionically flat hypercomplex manifold with quaternionic 
$S^{1}$-action whose first Chern class of the fixed-point set is not trivial as follows.  
Let $(M,H=(I_{1},I_{2},I_{3}))$ be a hypercomplex manifold. 
Consider $\mu=f I_{1} \otimes e \in \Gamma(Q \otimes L^{s_{0}})$ and calculate $\nabla^{0} \mu$
with respect to the Obata connection $\nabla^{0}$ of $H$. 
We have  
\begin{align*}
\nabla^{0} \mu=(df \otimes I_{1}+s_{0}f \theta_{L} \otimes I_{1}) \otimes e. 
\end{align*}
If $f$ is a nowhere vanishing function and $\mu$ is a twistor equation, then 
\[ \theta_{L}=-\frac{1}{s_{0}} d \log |f|. \]  
In particular, if the induced connection on $L$ from the Obata connection 
$\nabla^{0}$ has a parallel non-zero section, then $\mu=c I_{1}$ is a twistor function, 
where $c$ is non-zero constant. Thus $\mu^{-1}(0)=\emptyset$. 
This is true in hyperK\"ahlerian cases.  
Let $T^{\ast} \mathbb{C}P^{n}$ be the cotangent bundle of 
the complex projective space $\mathbb{C}P^{n}$. 
It is known that $T^{\ast} \mathbb{C}P^{n}$ carries the hyperK\"ahler structure 
(Calabi metric) with $W \neq 0$. 
Considering the $S^{1}$-action given by the scalar multiplication on the fibers,   
the fixed-point set of this action is $0$-section, which is $\mathbb{C}P^{n}$. See also Example \ref{ex_gr2}.

\subsection{The quaternionic projective space--Pontecorvo's example--}\label{sec_hp_n}
Let $\mathbb{H}P^{n}$ be the (right-)quaternionic projective space 
with its standard quaternionic structure $Q$. 
It is well known that $(\mathbb{H}P^{n},Q)$ admits the standard 
quaternionic K\"ahler metric $g$. 
We denote by $\mathrm{Aut}(Q)$ (resp.\ $\mathrm{Aut}(g)$) 
the group of quaternionic transformations 
(resp.\ isometries) of $\mathbb{H}P^{n}$. 
It is known that 
\[
\mathrm{Aut}(g)=\mathrm{Sp}(n+1)/\mathbb{Z}_{2} 
\subsetneq 
\mathrm{Aut}(Q)
=\mathrm{PGL}(n+1,\mathbb{H})
=\mathrm{GL}(n+1,\mathbb{H})/\mathbb{R}^{\ast}.
\]
See, for example, \cite{AM2}. 
Thus, in general, a quaternionic $S^{1}$-action on 
$(\mathbb{H}P^{n},Q)$ is not necessarily isometric 
with respect to $g$. 
To determine the fixed-point sets of a quaternionic 
$S^{1}$-action on $\mathbb{H}P^{n}$, 
we therefore work independently of the metric $g$. 
The following lemma is well known, 
however, we include an alternative and more elementary proof for completeness.

\begin{lemma}\label{q_inv}
Let $(M,Q)$ be a quaternionic manifold. 
If $N$ is a $Q$-invariant submanifold, then $N$ is totally geodesic for any quaternionic connection 
$\nabla$.
\end{lemma}
\begin{proof}
Let $\nabla$ be any quaternionic connection of $Q$ and  
set $h_{x}(X,Y):=[(\nabla_{X}Y)_{x}] \in T_{x}M/T_{x}N$ at $x \in N$ for all $X$, $Y \in \Gamma(TN)$. 
We have $h(X,I_{\alpha}Y)=I_{\alpha}h(X,Y)$ for any $\alpha \in \{1,2,3 \}$. 
Here we note that $I_{\alpha}[\xi]:=[I_{\alpha} \xi] \in T_{x}M/T_{x}N$ for $\xi \in T_{x}M$. 
Since $\nabla$ is torsion-free, 
$h$ is symmetric. Therefore, it holds that 
$h(I_{\alpha}X,I_{\alpha}Y)=-h(X,Y)$ for any $\alpha \in \{1,2,3 \}$. 
This implies that  
\[ h(I_{\gamma}X,I_{\gamma}Y)=h(I_{\alpha}I_{\beta}X,I_{\alpha}I_{\beta}Y)
=-h(I_{\beta}X,I_{\beta}Y)=h(X,Y) \]
for any cyclic permutation $(\alpha,\beta,\gamma)$ of $(1,2,3)$. 
Hence $h$ vanishes, that is, $\nabla_{X}Y \in \Gamma(TN)$ for all $X,Y \in \Gamma(TN)$. 
\end{proof}

 In particular, the induced quaternionic structure and connection on $N$ makes 
$N$ a quaternionic manifold.

\begin{definition}
A quaternionic manifold $(M,Q)$ is said to be {\cmssl complete}
if there exists a quaternionic connection $\nabla$ 
such that $(M,\nabla)$ is geodesically complete. 
Here, $(M,\nabla)$ is called $\nabla$-geodesically complete 
if every $\nabla$-geodesic is defined on the entire real line $\mathbb{R}$.
\end{definition}

The quaternionic projective space $\mathbb{H}P^{n}$ with standard quaternionic structure 
is complete. Indeed, the Levi-Civita connection of the standard Riemannian metric is a quaternionic connection.

\begin{lemma}\label{t_geodesic}
Let $N_{1}$ and $N_{2}$ are $Q$-invariant submanifolds in a complete quaternionic manifold $(M,Q)$. 
Assume that there exists a point $x \in N_{1} \cap N_{2}$ such that $T_{x}N_{1}=T_{x}N_{2}$. 
If $N_{1}$ is connected and $N_{2}$ is geodesically complete for the induced quaternionic connection, 
then $N_{1}$ is an open submanifold in $N_{2}$. 
\end{lemma}

\begin{proof}
The proof is similar as in 14. Lemma in p.105 of \cite{O}. 
By Lemma \ref{q_inv}, $N_{1}$ and $N_{2}$ are totally geodesic submanifolds for 
any quaternionic connection. Because $(M,Q)$ is complete, there exists  
a quaternionic connection $\nabla$ 
for which $(M,\nabla)$ is $\nabla$-geodesically complete. 
Let $\gamma$ be a geodesic on $N_{1}$ starting from $x$ to $y$ with respect to the induced connection from 
$\nabla$. 
So $\gamma$ is a geodesic on $(M,\nabla)$ and $\gamma^{\prime}(0) \in T_{x}N_{2}$.
Therefore, $\gamma$ is a geodesic on $N_{2}$ since $N_{2}$ is totally geodesic and $\nabla$-complete. 
Since the parallel translation of $T_{x}N_{1}=T_{x}N_{2}$ on $(M,\nabla)$ along 
$\gamma$ gives $T_{y}N_{1}=T_{y}N_{2}$. 
This shows $y \in N_{2}$. By repeating this argument and connected assumption for $N_{1}$, 
we have $N_{1} \subset N_{2}$.    
\end{proof}

Now we can characterize the fixed-point sets of a quaternionic $S^{1}$-action on $\mathbb{H}P^{n}$. 

\begin{theorem}\label{fix_hp_n}
If $S^{1}$ acts on $\mathbb{H}P^{n}$ preserving the quaternionic structure, then 
each connected component $F$ of the fixed-point set 
is one of the following: 
\begin{enumerate}
\item an isolated point, 
\item a submanifold diffeomorphic to $\mathbb{C}P^{l}$ with $f_{Q} \neq 0$ 
which is a transversally totally complex submanifold contained in 
$\mathbb{H}P^{l} \subset \mathbb{H}P^{n}$, 
\item a submanifold diffeomorphic to $\mathbb{H}P^{k}$ with $f_{Q}=0$. 
\end{enumerate}
\end{theorem}

\begin{proof}
By Theorem \ref{thm_fixpts}, 
$F$ is either a quaternionic or a transversally complex and totally geodesic submanifold. 
We first consider the case where $F$ is a quaternionic submanifold in $M$. 
We may assume that there exists a point $x \in F$ such that 
$T_{x}F=T_{x} \mathbb{H}P^{k}$, where $4k=\dim F$. 
By Lemma \ref{t_geodesic}, $F$ is an open submanifold in $\mathbb{H}P^{k}$. 
Since $F$ is closed, it follows that $F=\mathbb{H}P^{k}$. 

Next, we consider 
the case where $F$ is transversally complex. 
Set $V=TF \oplus I_{2} TF(=TF \oplus I_{3} TF)$, which is a $Q$-invariant subbundle of $TM|_{F}$ over $F$. 
Since $F$ is totally geodesic for a quaternionic connection, $V$ is parallel with respect to 
this connection, and hence with respect to any quaternionic connection (Remark \ref{rem_para}). 
We may assume that there exists $x \in F$ such that 
$V_{x} = T_{x}\mathbb{H}P^{l}$, where $4l=\mathrm{rank}\,  V$. 
Take a geodesic $\gamma$ on $F$ from $x$ to $y$. 
Then $\gamma$ is a geodesic in $\mathbb{H}P^{n}$, and hence, 
it is geodesic in $\mathbb{HP}^{l}$, and hence we obtain  
$V_{y}=T_{y}\mathbb{H}P^{l}$. This shows that $F$ is contained in $\mathbb{H}P^{l}$, where 
it is a (totally geodesic and) transversally complex submanifold with $\dim F=2l$. 
Finally, we show that $F=\mathbb{C}P^{l}$.  
We consider the case of $l \geq 2$. 
By Theorem \ref{thm_fixpts}, 
$F$ is a transversally totally complex submanifold and it is totally geodesic 
with respect to a quaternionic connection. 
This means that the second fundamental form of $F$ with respect to this quaternionic connection 
vanishes. 
Consequently, the $(2,0)+(0,2)$-part of the second fundamental form of $F$ with respect to 
the standard quaternionic connection on $\mathbb{H}P^{l}$ vanishes.     
By Theorem 6.7 in \cite{Tsu}, it follows that 
$F$ is an open part of $\mathbb{C}P^{l}$. 
Since $F$ is closed, we conclude that $F=\mathbb{C}P^{l}$. 
When $l=1$, $\mathbb{H}P^{1} \cong S^{4}$ with the conformal structure $[g_{0}]$, where 
$g_{0}$ is standard metric on $S^{4}$. Note
that a quaternionic structure is a conformal structure for four dimensional cases. 
Since $F$ is totally geodesic for a metric in $[g_{0}]$,  
$F$ is totally umbilic with respect to the metric $g_{0}$, hence $F$ is an open part of 
the 2-dimensional sphere of a certain radius, which is isomorphic to  
$\mathbb{C}P^{1}$. The closedness for $F$ implies that $F=\mathbb{C}P^{1}$.  
\end{proof}

\begin{remark}
{\rm 
For a transversally totally complex submanifold 
of half dimension in a quaternionic manifold, 
the $(2,0)+(0,2)$-part of the second fundamental form 
is independent of the choice of quaternionic connection. 
See \cite[Proposition 3.7]{Tsu}.
}
\end{remark}

As a consequence of Corollary \ref{q_flat_fix_qcm} and Theorem \ref{fix_hp_n}, we have the following.

\begin{theorem}\label{HP_Fix}
Let $(\mathbb{H}P^{n},Q)$ be the quaternionic projective space with the standard quaternionic structure 
$Q$ and an $S^{1}$-action preserving $Q$, and let $\mu$ be a twistor function. 
We assume that the complex structure $I$ induced from $\mu$ is $S^{1}$-invariant. 
Let $F$ be a connected component of the fixed-point set of the $S^{1}$-action with $f_{Q} \neq 0$ and 
$\mu^{-1}(0) \cap F =\emptyset$. 
Then 
$\mathbb{H}P^{n} \setminus \mu^{-1}(0)$ does not admit hypercomplex structures $(I_{1}=I,I_{2},I_{3})$
on any open set $U$ containing $F$.  
In particular, if $\dim F \geq 4$, then the ${\rm SO(3)}$-bundle
$S(U)$ is not trivial for any open set 
$U \subset \mathbb{H}P^{n} \setminus \mu^{-1}(0)$ such that $F \subset U$. 
Consequently, if $\dim F \geq 4$, then 
a connected component of $\mathbb{H}P^{n} \setminus \mu^{-1}(0)$ containing $F$ with $f_{Q} \neq 0$ admits no hypercomplex structure $H$ 
such that $Q=\langle H \rangle$.
\end{theorem}

\begin{proof}
It follows from $c_{1}(F)=c_{1}(\mathbb{C}P^{m})\neq 0$.  
\end{proof}

To illustrate Pontecorvo's example 
$\mathrm{SO}^{\ast}(2n+2)/\mathrm{SO}^{\ast}(2n) \times \mathrm{SO}^{\ast}(2)$ 
from our viewpoint, we consider the lifted map of a twistor function to the Swann bundle which carries the hypercomplex structure, which has been studied in Section \ref{sec_swann}. 
Consider the quaternionic $S^{1}$-action on $\mathbb{H}P^{n}$ given by 
\begin{align}\label{act_4_18}
e^{i \theta}[z]:=[e^{i \theta}z],  
\end{align}
where $[z] \in \mathbb{H}P^{n}$. Note that $e^{i \theta}(u+jv)=e^{i \theta}u+j e^{-i \theta}v$
for $u+jv \in \mathbb{C} \oplus \mathbb{C}=\mathbb{H}$. 
Take the associated pair $(e,\nu^{g})$, 
where $\nu^{g}$ is the volume form of the standard metric $g$ on 
$\mathbb{H}P^{n}$. 
We consider the twistor function $\mu=f_{Q}^{\nabla^{g}} \otimes e$ 
on $\mathbb{H}P^{n}$, where $\nabla^{g}$ is the Levi-Civita connection of $g$  
and $X$ generates the $S^{1}$-action. 
The vector field $\hat{X}$ which generates   
the lifted action to $\mathbb{H}^{n+1} \setminus \{ 0 \}$ is given by 
\begin{align*}
\hat{X}_{(z_{1},\dots, z_{n+1})} 
&= (iz_{1},\dots, iz_{n+1})=(iu_{1},\dots, iu_{n+1},-iv_{1},\dots,-iv_{n+1}) \\
&= \frac{i}{2} \sum_{l=1}^{n+1} \left( u_{l} \frac{\partial}{\partial u_{l}} - \bar{u}_{l} \frac{\partial}{\partial \bar{u}_{l}}
-v_{l} \frac{\partial}{\partial v_{l}} + \bar{v}_{l} \frac{\partial}{\partial \bar{v}_{l}} \right). 
\end{align*}
By Lemmas \ref{qh_moment}, \ref{5_14} and Proposition \ref{5_15}, the zero set of $\hat{\mu}$ coincides with 
the zeros of 
$\theta(\hat{X})=(\theta_{1}(\hat{X}),\theta_{2}(\hat{X}),\theta_{2}(\hat{X}))$, where 
$(\theta_{1},\theta_{2},\theta_{3})$ is the connection form of the ${\rm SO(3)}$-bundle $S(M)$ induced from $\nabla^{g}$:
\begin{align*} 
\theta_{1} &= 
r^{-2} {\rm Im} (\sum_{l=1}^{n+1} \bar{u}_{l} du_{l} + \bar{v}_{l} dv_{l}), \\
\theta_{2} &= 
r^{-2} {\rm Re} (\sum_{l=1}^{n+1} {u}_{l} dv_{l} - v_{l} du_{l}), \\
\theta_{3} &= 
r^{-2} {\rm Im} (\sum_{l=1}^{n+1} {u}_{l} dv_{l} - v_{l} du_{l}),
\end{align*}
where $r^{2}=\sum_{l=1}^{n+1} (\| u_{l} \|^{2} + \| v_{l} \|^{2})$. Then the zero set of $\hat{\mu}$ is given by
\[ \hat{\mu}^{-1}(0)=\{ (u,v) \in \mathbb{C}^{n+1} \oplus \mathbb{C}^{n+1} \mid 
    {}^{t} u\,v=0 \,\,\, \mbox{and} \,\,\, \| u \| = \| v \|  \}. \] 
The fixed-point set of the $S^{1}$-action on $\mathbb{H}P^{n}$ is $\mathbb{C}P^{n}$, which corresponds to 
$V^{\prime}:=\{ (u,0) \in \mathbb{C}^{n+1} \oplus \mathbb{C}^{n+1} \mid u \neq 0 \} 
\subset \mathbb{H}^{n+1} \setminus \{ 0 \}$. 
By Lemma \ref{5_11}, a connected component containing $V^{\prime}$ 
of the compliment of the zeros $\hat{\mu}^{-1}(0)$ is 
\[ \{ (u,v) \in \mathbb{C}^{n+1} \oplus \mathbb{C}^{n+1} \mid  {}^{t}u\,v=0 \,\,\, 
\mbox{and} \,\,\, \| u \| - \| v \| >0 \},  \]
which is 
$\mathrm{SO}^{\ast}(2n+2)/\mathrm{SO}^{\ast}(2n) \times \mathrm{SO}^{\ast}(2)$. 
In fact, using the change of coordinates 
\begin{align*}
z_{i} =\frac{1}{2} (u_{i}+v_{i}), \,\,\, z_{n+1+i}=-\frac{\sqrt{-1}}{2} (u_{i}-v_{i}) \,\,\, (i=1,\dots,n+1),
\end{align*}
we have 
\begin{align*}
{}^{t} z \, z =  {}^{t}u\,v, \,\,\, \| u \|^{2}=\| z \|^{2}+2 \sum_{i=1}^{n+1} \mathrm{Im} \, z_{i}\bar{z}_{n+1+i}, \,\,\,
 \| v \|^{2}=\| z \|^{2}-2 \sum_{i=1}^{n+1} \mathrm{Im} \, z_{i} \bar{z}_{n+i}. 
\end{align*}
See \cite[Proposition 3.8 and Theorem 3.9]{P} and \cite[Example 2]{Hit}. 
Therefore we summarize as follows.

\begin{example}[Pontecorvo's example \cite{P}]\label{ex_qp1}
{\rm 
The manifold
\[
M'=\mathrm{SO}^{\ast}(2n+2)/\mathrm{SO}^{\ast}(2n) \times \mathrm{SO}^{\ast}(2)
\]
is a quaternionic complex manifold with the twistor function 
$\mu=f_{Q}^{\nabla^{g}}|_{M'}$. The induced compatible complex structure $I$ by $\mu$ 
is not locally hypercomplex. 
Suppose that for any $x \in M^{\prime}$ there exists a neighborhood $U$ of $x$ 
on which $(I_{1}=I,I_{2},I_{3})$ defines a hypercomplex structure. 
Then there exists the Obata connection $\nabla^{0}$ on $U$ 
associated with $(I_{1}=I,I_{2},I_{3})$. 
Since $\nabla^{\mu}=\nabla^{0}$ on $U$ for each $x \in M^{\prime}$, the $\mu$-connection induces 
a flat connection on $Q$. 
As $M^{\prime}$ is simply connected, $I_{2}$ and $I_{3}$ extend to $M^{\prime}$ 
by the parallel translation with respect to $\nabla^{\mu}$. 
Hence the extended frame $(I_{1},I_{2},I_{3})$ defines a hypercomplex structure on $M^{\prime}$, 
contradicting Theorem \ref{HP_Fix}. 
In \cite{P}, it is shown that the holonomy group of the canonical connection 
of the (pseudo-Riemannian) symmetric space $M'$ is not contained in 
$\mathrm{SL}(n,\mathbb{H})$, which also implies that 
$(I_{1}=I,I_{2},I_{3})$ is not locally hypercomplex. 
See also \cite[Example 3.10]{CCG}.
}
\end{example}

\subsection{The complex two-plane Grassmann manifold}\label{sec_weight_grass}

We present a further example obtained as the quaternionic complex quotient of $\mathbb{H}P^{n}$. 
The example given here is not quaternionically flat, but by using the geometry of fixed-point sets 
(Proposition \ref{QK_Fix})
similar to the previous section, we can show the non-existence of a hypercomplex structure. 
We refer to \cite{J} for the quaternionic complex quotient and for 
interesting four-dimensional examples. 

Consider the $S^{1}$-action 
\eqref{act_4_18} and 
the twistor function $\mu$ on $\mathbb{H}P^{n}$ as in Example \ref{ex_qp1}. 
Then it holds that 
\begin{align*}
\mu^{-1}(0)=\{ [z]=[u+jv] \in \mathbb{H}P^{n} \mid {}^{t} u\,v=0 \,\,\, \mbox{and} \,\,\, \| u \| = \| v \|  \}.
\end{align*}
and
\begin{align*}
\mu^{-1}(0) / S^{1}=Gr(2,n+1),
\end{align*}
where $Gr(2,n+1)$ is the complex two-plane Grassmann manifold. 
Note that the quaternionic structure on $Gr(2,n+1)$ is not flat, 
that is, $W \neq 0$, 
as follows from the curvature formulas in \cite{Ber}.
We consider the twistor function $\mu^{\circ}$ on $\mathbb{H}P^{n}$ 
whose lifted map $\hat{\mu}^{\circ}$ to the Swann bundle 
is given by 
$\hat{\mu}^{\circ}_{1}=\| u_{1} \|^{2}-\| v_{1} \|^{2}$ and 
$\hat{\mu}^{\circ}_{2}+i \mu^{\circ}_{3} = 2i u_{1} v_{1}$. 
Since $\mu^{\circ}$ is invariant 
under the action \eqref{act_4_18}, $\mu^{\circ}$ descends to a twistor function on the quotient space $Gr(2,n+1)$. 
Therefore $\mathbb{H}P^{n} \setminus \{ [z] \mid z_{1}=0 \}$ is a quaternionic complex manifold 
with the twistor function $\mu^{\circ}$. 
In particular, on $\mu^{-1}(0) \setminus \{ [z] \in \mathbb{H}P^{n} \mid z_{1}=u_{1}+jv_{1} = 0 \}$, 
the twistor function $\mu^{\circ}$ has no zero points. 
Note that the quotient of  
$(\mu^{-1}(0) \setminus \{ [z] \in \mathbb{H}P^{n} \mid z_{1}= 0 \})$ by the $S^{1}$-action 
\eqref{act_4_18}
is a quaternionic complex quotient space 
with the twistor function induced by $\mu^{\circ}$ 
in the sense of \cite{J}. 
It is known that there is no compatible almost complex structure on $Gr(2,n+1)$. 
However, there exists a compatible complex structure on 
\begin{align*}
Gr(2,n+1) \setminus Gr(2,n). 
\end{align*}
Indeed, this open subset carries the quaternionic complex structure 
induced by $\mu^{\circ}$. 
Nevertheless, this compatible complex structure cannot be extended 
to a hypercomplex structure as Example \ref{ex_gr1}. 
To see this, we show the following proposition in the compact setting 
in contrast with the case of $T^{\ast} \mathbb{C}P^{n}$
endowed with the Calabi metric (see Section \ref{sec_mu_conn}).

\begin{proposition}\label{QK_Fix}
Let $(M,Q,g)$ be a compact quaternionic K\"ahler manifold 
with non-zero scalar curvature. 
Suppose that $S^{1}$ acts on $M$ isometrically, 
and let $F$ be a connected component of the fixed-point set 
such that $f_{Q} \neq 0$ and $F \cap \mu^{-1}(0)=\emptyset$, 
where $\mu$ is the quaternionic K\"ahler moment map associated with the $S^{1}$-action. 
Then $M \setminus \mu^{-1}(0)$ 
does not admit a hypercomplex structure $(I_{1}=I,I_{2},I_{3})$
on any open set $U$ containing $F$. 
In particular, if $\dim F \geq 4$, then the ${\rm SO}(3)$-bundle $S(U)$ is not trivial 
for any open set $U \subset M \setminus \mu^{-1}(0)$ with $F \subset U$. 
Consequently, if $\dim F \geq 4$, then 
a connected component of $M \setminus \mu^{-1}(0)$ containing $F$ with $f_{Q} \neq 0$ admits no hypercomplex structure $H$ 
such that $Q=\langle H \rangle$.
\end{proposition}

\begin{proof}
Let $\nabla^{g}$ be the Levi-Civita connection of $g$. 
Since connection $\nabla^{g}$ and $\nabla^{\mu}$ are quaternionic connections with respect to $Q$, we have 
\[ R^{\nabla^{g}}-R^{B^{\nabla^{g}}}=R^{\nabla^{\mu}}-R^{B^{\nabla^{\mu}}}(=W). \]  
By Corollary \ref{3.22}, we see that $\nabla^{g}_{X}Y=\nabla^{\mu}_{X}Y \in \Gamma(TF)$ for any 
$X$, $Y \in \Gamma(TF)$. 
Therefore, it holds that 
\[ p(R^{B^{\nabla^{g}}}_{X,Y}Z)=p(R^{B^{\nabla^{\mu}}}_{X,Y}Z) \]
for any $X$, $Y$, $Z \in \Gamma(TF)$
from Remark \ref{rem_para}, where $p$ is the projection onto the tangent component.  
Using Corollary \ref{cor_42}, we have
\begin{align*}
   & 4n \big( Ric^{\nabla^{g}}(Y,Z)X-Ric^{\nabla^{g}}(X,Z)Y
   -Ric^{\nabla^{g}}(Y,I_{1}Z)I_{1}X+Ric^{\nabla^{g}}(X,I_{1}Z)I_{1}Y \\
   &-2 Ric^{\nabla^{g}}(I_{1}X,Y)I_{1}Z \big) \\
 =& 4(n+2) \big( Ric^{\nabla^{\mu}}(Y,Z)X-Ric^{\nabla^{\mu}}(X,Z)Y-2Ric^{\nabla^{\mu}}(I_{1}X,Y)I_{1}Z \\ 
   &  -Ric^{\nabla^{\mu}}(Y,I_{1}Z)I_{1}X +Ric^{\nabla^{\mu}}(X,I_{1}Z)I_{1}Y \big)
\end{align*}
$X$, $Y$, $Z \in \Gamma(TF)$. Taking the trace in the slot of $X$, we have 
\begin{align}\label{4_20_ric}
n Ric^{\nabla^{g}}(Y,Z)=(n+2)Ric^{\nabla^{\mu}}(Y,Z), 
\end{align}
which means that 
\[  \iota^{\ast} c_{1}(M \setminus \mu^{-1}(0))=\frac{\tau }{8 \pi (n+2)} [ \omega_{F} ] \]  
where $\tau$ is the scalar curvature of $g$, $\omega_{F}$ is the K\"ahler form on $F$   
and $\iota:F \to M \setminus \mu^{-1}(0)$ is the inclusion. 
Since $F$ is a compact K\"ahler manifold due to \cite[Proposition 3.5]{B}, then $c_{1}(M \setminus \mu^{-1}(0)) \neq 0$. 
\end{proof}

\begin{remark}
{\rm 
Proposition \ref{local_obst} implies that for each point 
$x \in F$, any neighborhood $U \subset M$ of $x$
does not admit a hypercomplex structure $(I_{1}=I,I_{2},I_{3})$ on $U$
by \eqref{4_20_ric} provided that the $S^{1}$-action is isometric and 
$\tau \neq 0$. 
}
\end{remark}

\begin{remark}
{\rm 
If $(M,Q,g)$ is a simply connected compact quaternionic K\"ahler manifold 
with non-zero scalar curvature and $\dim M >4$ which is not $\mathbb{H}P^n{}$, then 
any quaternionic $S^{1}$-action is isometric \cite{AM2}. Since the action 
\eqref{act_4_18} on $\mathbb{H}P^{n}$ is isometric, 
the argument in Example \ref{ex_qp1} also follows from Proposition \ref{QK_Fix}.
}
\end{remark}

%
%
%

We now present examples illustrating the obstruction 
to the existence of a hypercomplex structure 
of the form $(I_{1}=I,I_{2},I_{3})$.

\begin{example}\label{ex_gr1}
{\rm 
The manifold $Gr(2,n+1) \setminus Gr(2,n)$ 
carries a quaternionic complex structure induced by 
the twistor function $\mu^{\circ}$. 
Note that $Gr(2,n)$ is the $S^{1}$-quotient of 
$\mu^{-1}(0) \cap \{ [z] \in \mathbb{H}P^{n} \mid z_{1}=0 \}$. 
Consider the $S^{1}$-action as in Example \ref{cpx_gr}. 
Then $\mu^{\circ}$ is $S^{1}$-invariant, and the fixed-point set 
contained in $Gr(2,n+1) \setminus Gr(2,n)$ is $\mathbb{C}P^{n-1}$. 
By Proposition \ref{QK_Fix}, 
$Gr(2,n+1) \setminus Gr(2,n)$ does not admit 
a hypercomplex structure $(I_{1}=I,I_{2},I_{3})$ 
on any open set $U$ containing $\mathbb{C}P^{n-1}$.
}
\end{example}



Finally, 
in order to describe the cotangent bundle $T^{\ast} \mathbb{C}P^{n-1}$ of the complex projective space 
as a quaternionic complex quotient, 
we consider another $S^{1}$-action 
\begin{align}\label{action}
e^{i \theta} \cdot [z] 
&:=[z_{1} : e^{i \theta} z_{2} : \cdots: e^{i \theta} z_{n+1}] 
\end{align}
and 
the twistor function $\mu$ on $\mathbb{H}P^{n}$ as in Example \ref{ex_qp1}.

\begin{example}\label{ex_gr2}
{\rm 
The singular set of the quotient space $\mu^{-1}(0)/_{(0,1,\dots,1)} S^{1}$ by the action \ref{action} is $Gr(2,n)=(\mu^{-1}(0) \cap \{ [z] \mid z_{1}=0 \})/ S^{1}$. 
Consider 
$\mathbb{C}^{n} \times \mathbb{C}^{n}$ defined by $(z,w) \mapsto (\lambda z, \bar{\lambda} w)$ 
for $\lambda \in S^{1}$, and the subset 
\[
{\cal L}:=\{ (z, w) \in \mathbb{C}^{n} \times \mathbb{C}^{n} 
\mid \| z \|^{2} - \| w \|^{2} - 1 = 0,\; 
\langle z,w \rangle = 0 \}.
\]
It is well known that 
$T^{\ast} \mathbb{C}P^{n-1} = {\cal L} / S^{1}$ 
(see \cite{Hit2}). 
We define the following maps:
\begin{align*}
& \iota: {\cal L} \to (\mathbb{C}^{n+1} \setminus \{0\}) \times (\mathbb{C}^{n+1} \setminus \{0\}); 
\quad (z,w) \mapsto ((1,w), (0,z)),  \\
& f:(\mathbb{C}^{n+1} \setminus \{0\}) \times (\mathbb{C}^{n+1} \setminus \{0\}) 
\to \mu^{-1}(0) ;
\quad (u,v) \mapsto [u + jv]. 
\end{align*}
Note that 
$\mathrm{Im}(f \circ \iota) 
= \mu^{-1}(0) 
\setminus \{ [z] \in \mathbb{H}P^{n} \mid z_{1}= 0 \}.$
We now define 
\[
\phi : {\cal L} / S^{1} \to \mu^{-1}(0)/_{(0,1,\dots,1)} S^{1}
\]
by $\phi([(z,w)]) := p(f(\iota(z,w)))$, where $p:\mu^{-1}(0) \to \mu^{-1}(0)/_{(0,1,\dots,1)} S^{1}$ 
is the quotient map. 
Then $\phi$ is well-defined and injective. 
Clearly,
$\mathrm{Im}\,\phi 
= (\mu^{-1}(0)/_{(0,1,\dots,1)} S^{1}) \setminus Gr(2,n)$
and 
$\phi(\{ 0 \text{-section} \}) 
\cong \mathbb{C}P^{n-1}$.
Therefore, the cotangent bundle 
\[ T^{\ast} \mathbb{C}P^{n-1} 
= (\mu^{-1}(0)/_{(0,1,\dots,1)} S^{1}) \setminus Gr(2,n)
\]
carries a quaternionic complex structure induced by 
the twistor function $\mu^{\circ}$ (this is a quaternionic complex quotient).   
Considering the $S^{1}$-action given by the scalar multiplication on the fibers,   
the fixed-point set of this action is $0$-section, which is $\mathbb{C}P^{n}$. 
}
\end{example}

%
%


\noindent
Kazuyuki Hasegawa \\
Faculty of teacher education \\
Institute of human and social sciences \\
Kanazawa university \\
Kakuma-machi, Kanazawa, \\
Ishikawa, 920-1192, Japan. \\
e-mail:kazuhase@staff.kanazawa-u.ac.jp


\begin{thebibliography}{99}


\bibitem{AM2}
D. Alekseevsky and S. Marchiafava, 
{\it Gradient quaternionic vector fields and a
characterization of the quaternionic projective space}, 
ESI 138 (1994).  


\bibitem{AM1}
D. Alekseevsky and S. Marchiafava,
{\it Quaternionic structures on a manifold and subordinated structures}, 
Ann. Mat. Pura Appl. (4) 171 (1996), 205-273. 


\bibitem{AMP}
D. Alekseevsky, S. Marchiafava, M. Pontecorvo,
{\it Compatible complex structures on almost quaternionic manifolds}, 
Trans. Amer. Math. Soc. 351 (1999), 997–1014.


\bibitem{B}
F. Battaglia, {\it Circle actions and Morse theory on quaternion-K\"ahler manifolds}, 
J. London Math. Soc. (2) 59 (1999), 345–358. 


\bibitem{Berger}
M. Berger, 
{\it Sur les groupes d'holonomie homogène des variétés à connexion affine et des variétés riemanniennes}, 
Bull. Soc. Math. France 83 (1955), 279–330.


\bibitem{Ber}
J. Berndt, 
{\it Riemannian geometry of complex two-plane Grassmannians}, 
Rend. Sem. Mat. Univ. Politec. Torino 55 (1997), 19–83.


\bibitem{BC}
A. Bor{\' o}wka and D. Calderbank, 
{\it Projective geometry and the quaternionic Feix-Kaledin construction}, 
Trans. Amer. Math. Soc. 372 (2019), 4729-4760.


\bibitem{Bor}
A. Bor{\' o}wka, 
{\it Complex quaternionic manifolds and  Projective geometry and C-projective structure}, 
Complex manifolds 12 (2025).  


\bibitem{Bre}
G. E. Bredon, Introduction to compact transformation groups, 
Pure and Applied Mathematics, Vol. 46. Academic Press, New York-London, 1972.





\bibitem{CCG}
I. Chrysikos, V. Cort{\' e}s and J. Gregorovi\v{c}, 
{\it Curvature of quaternionic skew-Hermitian manifolds and bundle constructions}.
Math. Nachr. 298 (2025), 87–112.

\bibitem{CH}
V. Cort{\' e}s and K. Hasegawa, 
{\it The quaternionic/hypercomplex-correspondence}, 
Osaka J. Math. 58 (2021), 213-238. 
\bibitem{Fu}
S. Fujimura, 
{\it $Q$-connections and their changes on an almost quaternion manifold}, 
Hokkaido Math. J. {\bf 5} (1976), 239-248. 




%


\bibitem{Hit}
N. Hitchin,
{\it Manifolds with holonomy  $\mathrm{U}^{\ast}(2m)$}, Rev. Mat. Complut. 27 (2014), 351–368.

\bibitem{Hit2}
N. Hitchin, 
{\it Metrics on moduli spaces}, 
Contemp. Math., 58, American Mathematical Society, 1986, 157–178.

\bibitem{J}
D. Joyce, 
{\it The hypercomplex quotient and the quaternionic quotient}, 
Math. Ann. 290 (1991), 323-340. 

\bibitem{K}
S. Kobayashi, Transformation groups in differential geometry, Springer Berlin, 1972. 


\bibitem{MS}
S. Merkulov and L.Schwach{\"o}fer, 
{\it Classification of irreducible holonomies of torsion-free affine connections}, 
Ann. of Math. (2) 150 (1999), 77–149.

\bibitem{O}
B. O'Neill, Semi-Riemannian geometry, 
Pure and Applied Mathematics, 103. Academic Press, Inc., New York, 1983.
 
\bibitem{PPS}
H. Pedersen, Y. Poon and A. Swann, 
{\it Hypercomplex structures associated to quaternionic manifolds}, 
Differential Geom. Appl. {\bf 9} (1998), 273-292.
 
 
 
 
\bibitem{P} 
M. Pontecorvo, 
{\it Complex structures on quaternionic manifolds}, 
Differential Geom. Appl. 4 (1994), 163-177.

\bibitem{P2} 
M. Pontecorvo, 
{\it On twistor spaces of anti-self-dual Hermitian surfaces}, 
Trans. Amer. Math. Soc. 331 (1992), 653–661.
 
\bibitem{S} 
S. Salamon, 
{\it Differential geometry of quaternionic manifolds}, 
Ann. Scient. {\' E}c. Norm. Sup. 19 (1986), 31-55.

\bibitem{Tsu}
K. Tsukada, 
{\it The Gauss maps of transversally complex submanifolds of a quaternion projective space}, 
Tohoku Math. J. (2) 73 (2021), 1–28. 
\end{thebibliography}
\end{document}